\documentclass[a4paper, reqno, 11pt]{amsart}
\usepackage[utf8]{inputenc}

\usepackage{amsthm}
\usepackage{amssymb}
\usepackage[american]{babel}
\usepackage{exscale}
\usepackage[mathscr]{euscript}
\usepackage{url}
\usepackage[all,ps,tips,tpic]{xy}
\usepackage{graphicx}
\usepackage{color}

\newtheorem{lem}{Lemma}[section]

\newtheorem{definition}[lem]{Definition}

\newtheorem{cor}[lem]{Corollary}
\newtheorem{thm}[lem]{Theorem}
\newtheorem*{thmA}{Theorem A}
\newtheorem*{thmB}{Theorem B}
\newtheorem*{thmC}{Theorem C}
\newtheorem*{thmD}{Theorem D}
\newtheorem{prop}[lem]{Proposition}

\theoremstyle{remark}
\newtheorem{rem}[lem]{Remark}

\DeclareMathOperator{\im}{im}

\newcommand{\tr}{\operatorname*{tr}}
\newcommand{\GL}{\mathrm{GL}}

\newdir{ >}{{}*!/-5pt/@{>}}
\SelectTips{cm}{10}

\begin{document}
\title[The Langlands-Kottwitz approach for some Shimura Varieties]{The Langlands-Kottwitz approach for some simple Shimura Varieties}
\author{Peter Scholze}
\begin{abstract}
We show how the Langlands-Kottwitz method can be used to determine the semisimple local factors of the Hasse-Weil zeta-function of certain Shimura varieties. On the way, we prove a conjecture of Haines and Kottwitz in this special case.
\end{abstract}

\date{\today}
\maketitle
\tableofcontents
\pagebreak

\section{Introduction}

The aim of this paper is to extend the method used in \cite{Scholze} for determining the zeta-function of the modular curve to the case of the unitary group Shimura varieties with signature $(1,n-1)$, considered e.g. in the book of Harris-Taylor, \cite{HarrisTaylor}. These are proper Shimura varieties for which endoscopy does not occur and for which particularly nice integral models exist by the theory of Drinfeld level structures.

Our first main result is of a general geometric nature. Let $\mathcal{O}$ be the ring of integers in a local field $K$. Assume $X / \mathcal{O}$ is a separated scheme of finite type. Then we denote by $X_{\eta^{\mathrm{ur}}}$ the base-change of $X$ to the maximal unramified extension $K^{\mathrm{ur}}$ of $K$ and by $X_{\mathcal{O}^{\mathrm{ur}}}$ the base-change to the ring of integers in $K^{\mathrm{ur}}$. Also let $X_{\overline{s}}$ be the geometric special fibre. We get maps $\iota: X_{\overline{s}}\longrightarrow X_{\mathcal{O}^{\mathrm{ur}}}$ and $j: X_{\eta^{\mathrm{ur}}}\longrightarrow X_{\mathcal{O}^{\mathrm{ur}}}$.

Now assume that $X$ is regular and flat of relative dimension $n$ over $\mathcal{O}$ and that the special fibre $X_s$ possesses a stratification $X_s = \bigcup_i \mathring{Z}_i$ into locally closed strata $\mathring{Z}_i$ whose closures $Z_i$ are regular. We call the $Z_i$ special fibre components. Let $c(Z)$ be the codimension of $Z$ in $X$, for any special fibre component $Z$ in $X$.

To any special fibre component $Z$, we associate a $\overline{\mathbb{Q}}_{\ell}$-vector space $W_Z$ together with maps $W_Z\longrightarrow W_{Z^{\prime}}$ whenever $c(Z)=c(Z^{\prime})+1$, by induction on $c(Z)$. If $c(Z)=1$, we simply set $W_Z=\overline{\mathbb{Q}}_{\ell}$. In general, we define $W_Z$ to be the kernel of the map
\[
\bigoplus_{\substack{Z\subset Z^{\prime}\\ c(Z^{\prime})=c(Z)-1}} W_{Z^{\prime}} \longrightarrow \bigoplus_{\substack{Z\subset Z^{\prime}\\ c(Z^{\prime})=c(Z)-2}} W_{Z^{\prime}}\ ,
\]
where in the case $c(Z)=2$, the target has to be replaced by $\overline{\mathbb{Q}}_{\ell}$. We make the following assumption:
\vspace{0.1in}

($\ast$) The sequence
\[
0\longrightarrow W_Z\longrightarrow \ldots \longrightarrow \bigoplus_{\substack{Z\subset Z^{\prime}\\ c(Z^{\prime})=i}} W_{Z^{\prime}} \longrightarrow \bigoplus_{\substack{Z\subset Z^{\prime}\\ c(Z^{\prime})=i-1}} W_{Z^{\prime}} \longrightarrow \ldots \longrightarrow \overline{\mathbb{Q}}_{\ell}\longrightarrow 0
\]
is exact for all special fibre components $Z$.

\begin{thmA} Assume that the condition ($\ast$) holds. Then for any $x\in X_s(\mathbb{F}_q)$ with geometric point $\overline{x}$ over $x$, there are canonical isomorphisms
\[
(\iota^{\ast}R^kj_{\ast} \overline{\mathbb{Q}}_{\ell})_{\overline{x}}\cong \bigoplus_{\substack{Z, c(Z)=k\\x\in Z}} W_Z(-k)
\]
for all $k$.
\end{thmA}

This generalizes Theorem A of \cite{Scholze}, and again makes use of Thomason's purity theorem, \cite{Thomason}, a special case of Grothendieck's purity conjecture.

It is worthwhile noting that this theorem implies the (nontrivial) fact that the groups
\[
(\iota^{\ast}R^kj_{\ast} \overline{\mathbb{Q}}_{\ell})_{\overline{x}}
\]
are pure (of weight $2k$). Also, the description is analogous to results of Brieskorn, \cite{Brieskorn}, Lemma 5, on the cohomology of complements of hyperplane arrangements. Indeed, one may think of the special fibre components of maximal dimension as hyperplanes, and the cohomology groups can be thought of as the cohomology groups of a small neighborhood of $\overline{x}$ in the complement of the geometric special fibre $X_{\overline{s}}$ in $X_{\mathcal{O}^{\mathrm{ur}}}$.

In order to formulate our other results, we need to introduce certain Shimura varieties. Fix an imaginary quadratic field $k$ and a central division algebra $D$ over $k$ of dimension $n^2$, with an involution $\ast$ of the second kind. Further, let an $\mathbb{R}$-algebra homomorphism
\[
h_0: \mathbb{C}\longrightarrow D_{\mathbb{R}}
\]
be given such that $h_0(z)^{\ast}=h_0(\overline{z})$ and such that the involution $x\longmapsto h_0(i)^{-1}x^{\ast}h_0(i)$ is positive. These data give rise to an algebraic group $\mathbf{G}$ over $\mathbb{Q}$ whose $R$-valued points are
\[
\mathbf{G}(R)=\{g\in (D\otimes_\mathbb{Q} R)^{\times} \mid g^{\ast}g\in R^{\times} \}\ ,
\]
and a homomorphism
\[
h: R_{\mathbb{C}/\mathbb{R}} \mathbb{G}_m\longrightarrow \mathbf{G}_{\mathbb{R}}\ .
\]
Taking any small compact open subgroup $K\subset \mathbf{G}(\mathbb{A}_f)$, one gets a smooth and proper Shimura variety $\mathrm{Sh}_K$ associated to $(\mathbf{G},h^{-1},K)$, canonically defined over $k$, although the reflex field may be even smaller, i.e. $\mathbb{Q}$.

We make the assumption that $\mathbf{G}\cong \mathbf{GU}(1,n-1)$. It is well-known that under these assumptions the associated Shimura varieties $\mathrm{Sh}_K$ have good integral models at primes $\mathfrak{p}$ above primes $p$ which split in $k$ and such that $D$ splits at $\mathfrak{p}$, if the level is given by a principal congruence subgroup. In fact, these integral models satisfy the hypotheses of Theorem A, as proved in Theorem \ref{SpecialFiberShimuraVariety} and Lemma \ref{CombSteinberg}.

For any $r\geq 1$ and any $h\in C_c^{\infty}(\GL_n(\mathbb{Z}_p))$, we define a function $\phi_h=\phi_{r,h}$ in the Hecke algebra of $\GL_n(\mathbb{Q}_{p^r})$, cf. Definition \ref{DefPhih}. Its values encode the semisimple trace of Frobenius on the vanishing cycles, cf. Theorem \ref{CalcSemisimpleTrace}.

Finally, let $f_{n,p}$ be the function of the Bernstein center of $\GL_n(\mathbb{Q}_p)$ such that for all irreducible admissible smooth representations $\pi$ of $\GL_n(\mathbb{Q}_p)$, $f_{n,p}$ acts by the scalar
\[
p^{\frac {n-1}2 r} \mathrm{tr}^{\mathrm{ss}}(\Phi_p^r|\sigma_{\pi})\ ,
\]
where $\sigma_{\pi}$ is the representation of the Weil group $W_{\mathbb{Q}_p}$ of $\mathbb{Q}_p$ with values in $\overline{\mathbb{Q}}_{\ell}$ associated to $\pi$ by the Local Langlands Correspondence, cf. \cite{HarrisTaylor}.

Assume that the Haar measures give maximal compact subgroups measure $1$.

\begin{thmB} For all $h\in C_c^{\infty}(\GL_n(\mathbb{Z}_p))$, the functions $\phi_h$ and $f_{n,p}\ast h$ have matching (twisted) orbital integrals.
\end{thmB}

This theorem furnishes a comparison between a purely geometric expression, $\phi_h$, and a purely automorphic expression, $f_{n,p}\ast h$. It is analogous to Theorem B in \cite{Scholze} for the case of modular curves, with some notable differences. Usually, one first compares $\phi_h$ with a function in the Bernstein center of $\GL_n(\mathbb{Q}_{p^r})$ and then invokes a base-change identity as in the case of a hyperspecial maximal compact subgroup. Our Theorem B does both steps at once.

To get the comparison with the twisted orbital integrals of a function in the Bernstein center of $\GL_n(\mathbb{Q}_{p^r})$, we prove the following base-change identity. In Section 2, we define a certain class of smooth group schemes $\mathcal{G}=\mathcal{G}_{M,I}$ over $\mathbb{Z}_p$, whose generic fibre is isomorphic to $\GL_n$. They include the case of principal congruence subgroups and parahoric subgroups. As a general piece of notation, we let $\mathcal{Z}(G,K)$ be the center of the Hecke algebra $\mathcal{H}(G,K)$ of compactly supported $K$-biinvariant functions on $G$, for any compact open subgroup $K$ of $G=\GL_n(\mathbb{Q}_p)$ or $G=\GL_n(\mathbb{Q}_{p^r})$.

\begin{thmC} Let $f\in \mathcal{Z}(\GL_n(\mathbb{Q}_p),\mathcal{G}(\mathbb{Z}_p))$ and $\phi\in \mathcal{Z}(\GL_n(\mathbb{Q}_{p^r}),\mathcal{G}(\mathbb{Z}_{p^r}))$ be given. For every tempered irreducible smooth representation $\pi$ of $\GL_n(\mathbb{Q}_p)$ with base-change lift $\Pi$, the set of invariants $\pi^{\mathcal{G}(\mathbb{Z}_p)}$ is nonzero only if $\Pi^{\mathcal{G}(\mathbb{Z}_{p^r})}$ is nonzero. Assume that for all $\pi$ with $\pi^{\mathcal{G}(\mathbb{Z}_p)}\neq 0$, the scalars $c_{f,\pi}$ resp. $c_{\phi,\Pi}$ through which $f$ resp. $\phi$ act on $\pi^{\mathcal{G}(\mathbb{Z}_p)}$ resp. $\Pi^{\mathcal{G}(\mathbb{Z}_{p^r})}$, agree: $c_{f,\pi} = c_{\phi,\Pi}$.

Then $f$ and $\phi$ have matching (twisted) orbital integrals.
\end{thmC}

\begin{rem} In fact, we prove this theorem for any local field of characteristic $0$ in place of $\mathbb{Q}_p$.
\end{rem}

This theorem generalizes, for the special case of the group $\GL_n$, the classical base-change fundamental lemma as well as the results of Haines about parahoric subgroups $\mathcal{G}$ in \cite{HainesBC}. The proof of Theorem C makes strong use of (unramified) base change for $\GL_n$ and hence does not generalize to arbitrary groups.

Using Theorem C, one arrives at an expression for the semisimple Lefschetz number as a sum of a volume factor times an orbital integral away from $p$ times a twisted orbital integral of a central function in a certain Hecke algebra at $p$, cf. Theorem \ref{HainesKottwitz}, procing a general conjecture of Haines and Kottwitz in this special case.

Our last main theorem calculates the {\it semisimple} local factor of the Hasse-Weil zeta-function of $\mathrm{Sh}_K/k$ in terms of automorphic $L$-functions. Recall that $\mathfrak{p}$ is a prime of $k$ lying above a rational prime $p$ that splits in $k$ and such that $D$ splits at $\mathfrak{p}$. In particular,
\[
\mathbf{G}_{\mathbb{Q}_p}\cong \GL_n\times \mathbb{G}_m\ .
\]

\begin{thmD} Let $K\subset \mathbf{G}(\mathbb{A}_f)$ be any small compact open subgroup of the form
\[
K=K_p^0\times \mathbb{Z}_p^{\times}\times K^p\subset \GL_n(\mathbb{Q}_p)\times \mathbb{Q}_p^{\times}\times \mathbf{G}(\mathbb{A}_f^p)\cong \mathbf{G}(\mathbb{A}_f)\ .
\]
As in \cite{KottwitzLambdaAdic}, p.656, there is a representation $r$ of the local $L$-group ${^L}\mathbf{G}_{k_\mathfrak{p}}$. Then the semisimple local factor of $\mathrm{Sh}_K$ at $\mathfrak{p}$ is given by
\[
\zeta_{\mathfrak{p}}^{\mathrm{ss}}(\mathrm{Sh}_K,s) = \prod_{\pi_f} L^{\mathrm{ss}}(s-\tfrac{n-1}{2},\pi_p,r)^{a(\pi_f)\mathrm{dim} \pi_f^K}\ ,
\]
where $\pi_f$ runs over irreducible admissible representations of $\mathbf{G}(\mathbb{A}_f)$ and where $a(\pi_f)\in\mathbb{Z}$ is given by
\[
a(\pi_f)=\sum_{\pi_{\infty}} m(\pi_f\otimes \pi_{\infty}) \tr(f_{\infty}|\pi_{\infty})\ .
\]
Here $m(\pi_f\otimes\pi_{\infty})$ is the multiplicity of $\pi_f\otimes \pi_\infty$ in $L^2(\mathbf{G}(\mathbb{Q})A_{\mathbf{G}}(\mathbb{R})^0\backslash \mathbf{G}(\mathbb{A}))$, where $A_{\mathbf{G}}$ is the split component of the center of $\mathbf{G}$. Finally, $f_{\infty}$ is the function defined by Kottwitz in \cite{KottwitzLambdaAdic}. It is up to sign a pseudo-character of a discrete series representation with trivial central and infinitesimal character.
\end{thmD}

\begin{rem} For a discussion of the semisimple local factor, we refer e.g. to \cite{HainesNgo}, cf. also \cite{Scholze}.

Note that under the assumptions, the local group $\mathbf{G}_{\mathbb{Q}_p}\cong \GL_n\times \mathbb{G}_m$. Therefore the local Langlands Correspondence for $\mathbf{G}_{\mathbb{Q}_p}$ is known and hence one can define the semisimple $L$-factor $L^{\mathrm{ss}}(s,\pi_p,r)$ as the semisimple $L$-factor $L^{\mathrm{ss}}(s,r\circ \sigma_{\pi_p})$ of the Weil group representation $\sigma_{\pi_p}$ associated to $\pi_p$ through the Local Langlands Correspondence, see e.g. \cite{HarrisTaylor}. However, unraveling all the definitions, one sees that one could define this $L$-factor directly, without appealing to the work of Harris-Taylor, cf. Remark \ref{DefSemisimpleTrace}.
\end{rem}

Note that Theorem D is essentially known by the work of Harris-Taylor, \cite{HarrisTaylor}. The emphasis of this paper lies in the method of proof. It is an extension of the method used by Kottwitz in \cite{KottwitzLambdaAdic} for a hyperspecial maximal compact level structure at $p$, and more classically by Langlands in \cite{Langlands} to prove a local-global compatibility statement for the cohomology of modular curves in the case that the local representation is not supercuspidal. In a similar spirit is the proof of Theorem 11.7 in \cite{Haines} in the case that $K_p$ is an Iwahori subgroup (but without restriction on the signature of $\mathbf{G}$) and the article of Haines and Rapoport, \cite{HainesRapoport}, for the case that the level structure at $p$ is given by the pro-unipotent radical of an Iwahori subgroup. Here, as in \cite{Scholze}, which discusses the case of $\GL_2$, we show that this method can be applied for arbitrarily small level structures at $p$.

In the paper \cite{ScholzeLLC}, we use the methods and results of this paper to give a new proof of the Local Langlands Correspondence for $p$-adic local fields, avoiding the use of the numerical Local Langlands Correspondence of Henniart, \cite{HenniartNumericalLLC}.

We now describe the content of the different sections. In Section 2, we prove the base-change identity, Theorem C. It makes strong use of base-change for $\GL_n$. Afterwards, we review in Section 3 the geometric results concerning the Shimura varieties, particularly their interpretation as moduli spaces of abelian varieties and a description of the special fibre. In Section 4, we recall some of the aspects of Kottwitz' work on the description of isogeny classes in the special fibre of these Shimura varieties, for a hyperspecial maximal compact level structure at $p$. The main geometric result, Theorem A, is then formulated and proved in Section 5 along with the application to the Shimura varieties in question. We remark that the result is closely related to the cohomology of Drinfeld's upper half plane, as determined by Schneider-Stuhler, \cite{SchneiderStuhler}. These results are then reinterpreted as orbital integrals in Section 6 to prove Theorem B, again making use of base-change for $\GL_n$. Finally, Section 7 concludes with the proof of Theorem D.

\begin{rem} Throughout the paper we fix a rational prime $\ell$ different from $p$ and an isomorphism $\overline{\mathbb{Q}}_{\ell}\cong \mathbb{C}$.
\end{rem}

{\bf Acknowledgments}. I thank my advisor M. Rapoport for introducing me to this area, his constant encouragement during the process of writing this paper, and his interest in these ideas.

\section{A base change identity}

Let $K$ be a local field of characteristic $0$ with ring of integers $\mathcal{O}_K$ and uniformiser $\varpi$ and let $L$ be an unramified extension of $K$ with ring of integers $\mathcal{O}_L$. Consider the group $G=\mathrm{GL}_n$ over $K$. We give a construction of compact open subgroups of $G(K)$ and $G(L)$.

Start with an $\mathcal{O}_K$-subalgebra $M$ of $M_n(\mathcal{O}_K)$ such that $M\otimes K = M_n(K)$. Further, take a (two-sided) ideal $I\subset M$ such that $M/I$ is finite. Then one gets the smooth group scheme $\mathcal{G}$ over $\mathcal{O}_K$ defined by
\[
\mathcal{G}_{M,I}(R) = \{ x\in I\otimes_{\mathcal{O}_K} R\mid g=1+x\in (M\otimes_{\mathcal{O}_K} R)^{\times} \} \ .
\]
In particular, we get compact open subgroups $\mathcal{G}_{M,I}(\mathcal{O}_L)$, resp. $\mathcal{G}_{M,I}(\mathcal{O}_K)$, of $G(L)$, resp. $G(K)$. We write $\mathcal{G}_M = \mathcal{G}_{M,M}$.

Let us give two examples. First, taking $M=M_n(\mathcal{O}_K)$ and $I=\varpi^mM_n(\mathcal{O}_K)$, we get the principal congruence subgroups
\[
\mathcal{G}_{M,I}(\mathcal{O}_K) = \ker(\GL_n(\mathcal{O}_K)\longrightarrow \GL_n(\mathcal{O}_K/\varpi^m\mathcal{O}_K) )\ .
\]
Second, one may fix a lattice chain
\[
\Lambda_0 = \mathcal{O}_K^n\subsetneq \Lambda_1\subsetneq \cdots \subsetneq \Lambda_k=\varpi^{-1}\mathcal{O}_K^n
\]
and consider the subalgebra
\[
M=\{m\in M_n(\mathcal{O}_K) \mid m\Lambda_i\subset \Lambda_i\ \mathrm{for\ all}\ i\}\subset M_n(\mathcal{O}_K)\ .
\]
Then $\mathcal{G}_M$ is the parahoric group scheme given by the lattice chain $(\Lambda_i)_i$.

We begin by proving the following comparison of conjugacy and $\sigma$-conjugacy classes.

\begin{prop}\label{SigmaConjClasses} There is a unique map $N$ from the set of $\sigma$-conjugacy classes in $\mathcal{G}_M(\mathcal{O}_L)/\mathcal{G}_{M,I}(\mathcal{O}_L)$ to the set of conjugacy classes in $\mathcal{G}_M(\mathcal{O}_K)/\mathcal{G}_{M,I}(\mathcal{O}_K)$, satisfying the requirement that if $\delta\in \mathcal{G}_M(\mathcal{O}_L)$ has the property $\delta\delta^{\sigma}\cdots\delta^{\sigma^{r-1}}\in \mathcal{G}_M(\mathcal{O}_K)$, then the $\sigma$-conjugacy class of $\delta$ is sent to the conjugacy class of $\delta\delta^{\sigma}\cdots\delta^{\sigma^{r-1}}$.

Moreover, the map $N$ is a bijection, and the size of the $\sigma$-centralizer of some element $\delta\in \mathcal{G}_M(\mathcal{O}_L)/\mathcal{G}_{M,I}(\mathcal{O}_L)$ equals the size of the centralizer of $N\delta\in \mathcal{G}_M(\mathcal{O}_K)/\mathcal{G}_{M,I}(\mathcal{O}_K)$.
\end{prop}

\begin{proof} Let $\gamma\in \mathcal{G}_M(\mathcal{O}_K)$. Consider the subalgebra $R_{\gamma} = \mathcal{O}_K[\gamma]\subset M$. Note that this is a finite free $\mathcal{O}_K$-algebra, so that the functor
\[
\mathcal{Z}_{\gamma}(R) = (R_{\gamma}\otimes_{\mathcal{O}_K} R)^{\times}
\]
describes a smooth commutative group scheme over $\mathcal{O}_K$.

\begin{lem} The norm homomorphism
\[\begin{aligned}
N : \mathcal{Z}_{\gamma}(\mathcal{O}_L)&\longrightarrow \mathcal{Z}_{\gamma}(\mathcal{O}_K)\\
\delta&\longmapsto \delta\delta^{\sigma}\cdots \delta^{\sigma^{r-1}}
\end{aligned}\]
is surjective.
\end{lem}

\begin{proof} The proof is identical to the proof of Lemma 3.4 of \cite{Scholze}.
\end{proof}

Hence, given $\overline{\gamma}\in \mathcal{G}_M(\mathcal{O}_K)/\mathcal{G}_{M,I}(\mathcal{O}_K)$, choose some lift $\gamma\in \mathcal{G}_M(\mathcal{O}_K)$ and $\delta\in \mathcal{Z}_{\gamma}(\mathcal{O}_L)$ with $N\delta = \gamma$, with reduction $\overline{\delta}\in \mathcal{G}_M(\mathcal{O}_L)/\mathcal{G}_{M,I}(\mathcal{O}_L)$. We claim that
\[
\{ x\in \mathcal{G}_M(\mathcal{O}_L)/\mathcal{G}_{M,I}(\mathcal{O}_L) \mid x^{-1}\overline{\delta} x^{\sigma} = \overline{\delta} \} = \{ x\in \mathcal{G}_M(\mathcal{O}_K)/\mathcal{G}_{M,I}(\mathcal{O}_K) \mid x^{-1}\overline{\gamma} x = \overline{\gamma}\} \ .
\]

Take $x\in \mathcal{G}_M(\mathcal{O}_L)/\mathcal{G}_{M,I}(\mathcal{O}_L)$ with $x^{-1}\overline{\delta} x^{\sigma} = \overline{\delta}$. Then $x^{-\sigma^i}\overline{\delta}^{\sigma^i} x^{\sigma^{i+1}} = \overline{\delta}^{\sigma^i}$ for all $i=0,\ldots,r-1$ and multiplying these equations gives
\[ x^{-1}N\overline{\delta} x = N\overline{\delta} \ ,\]
hence $x$ commutes with $\overline{\gamma}$, i.e. $x$ commutes with $\gamma$ modulo $I$. But then $x$ commutes with $\delta$ modulo $I$, because $\delta\in \mathcal{O}_L[\gamma]$, i.e. $x$ commutes with $\overline{\delta}$. Therefore $x^{-1}\overline{\delta} x = \overline{\delta} = x^{-1}\overline{\delta}x^{\sigma}$, and hence $x = x^{\sigma}$, whence $x\in \mathcal{G}_M(\mathcal{O}_K)/\mathcal{G}_{M,I}(\mathcal{O}_K)$. The other direction is clear.

This proves that the size of centralizers equals the size of $\sigma$-centralizers. Now a counting argument finishes the proof, cf. proof of Proposition 3.3 of \cite{Scholze}.
\end{proof}

We use this proposition to prove the following identity.

\begin{cor}\label{BCUnit} Let $f$ be a conjugation-invariant locally integrable function on $\mathcal{G}_{M}(\mathcal{O}_K)$. Then the function $\phi$ on $\mathcal{G}_{M}(\mathcal{O}_L)$ defined by $\phi(\delta)=f(N\delta)$ is locally integrable. Furthermore,
\[(e_{\mathcal{G}_{M,I}(\mathcal{O}_L)}\ast \phi)(\delta) = (e_{\mathcal{G}_{M,I}(\mathcal{O}_K)}\ast f)(N\delta)\]
for all $\delta\in \mathcal{G}_{M}(\mathcal{O}_L)$.
\end{cor}

\begin{proof} Assume first that $f$ is locally constant, say invariant by $\mathcal{G}_{M,I^{\prime}}(\mathcal{O}_K)$ for some ideal $I^{\prime}$. Of course, $\phi$ is then invariant under $\mathcal{G}_{M,I^{\prime}}(\mathcal{O}_L)$ as well and in particular locally integrable. The desired identity follows on combining the above Proposition for the ideals $I$ and $I^{\prime}$: Using it for $I$, we see that we may average over the $\sigma$-conjugacy class of $\delta$. Then we get the sum over all
\[
\delta^{\prime}\in \mathcal{G}_M(\mathcal{O}_L)/\mathcal{G}_{M,I^{\prime}}(\mathcal{O}_L)
\]
which are congruent to $\delta$ modulo $I$ of the averages of $\phi$, resp. $f$, over the $\sigma$-conjugacy class of $\delta^{\prime}$, resp. the conjugacy class of $N\delta^{\prime}$, which agree by the Proposition for $I^{\prime}$.

The corollary now follows by approximating $f$ by locally constant functions.
\end{proof}

Let tempered representations $\pi$, $\Pi$ of $G(K)$, $G(L)$, resp., be given. Further, let $\sigma\in \text{Gal}(L/K)$ be the lift of Frobenius. 

\begin{definition} In this situation, $\Pi$ is called a base-change lift of $\pi$ if $\Pi$ is invariant under $\emph{Gal}(L/K)$ and for some extension of $\Pi$ to a representation of $G(L)\rtimes \emph{Gal}(L/K)$, the identity
\[ \tr( Ng | \pi ) = \tr( (g,\sigma) | \Pi ) \]
holds for all $g\in G(L)$ such that the conjugacy class of $Ng$ is regular semisimple.
\end{definition}

It is known that there exist unique base-change lifts by the work of Arthur-Clozel, cf. \cite{ArthurClozel}. Write $\mathcal{Z}(G_1,G_2)$ for the center of the Hecke algebra of compactly supported $G_2$-biinvariant functions on $G_1$, for any parameters $G_1$, $G_2$ occuring in the sequel.

\begin{thm}\label{BaseChangeIdentity} Let $f\in \mathcal{Z}(G(K),\mathcal{G}_{M,I}(\mathcal{O}_K))$ and $\phi\in \mathcal{Z}(G(L),\mathcal{G}_{M,I}(\mathcal{O}_L))$ be given. For every tempered irreducible smooth representation $\pi$ of $G$ with base-change lift $\Pi$, the set of invariants $\pi^{\mathcal{G}_{M,I}(\mathcal{O}_K)}$ is nonzero only if $\Pi^{\mathcal{G}_{M,I}(\mathcal{O}_L)}$ is nonzero. Assume that for all $\pi$ with $\pi^{\mathcal{G}_{M,I}(\mathcal{O}_K)}\neq 0$, the scalars $c_{f,\pi}$ resp. $c_{\phi,\Pi}$ through which $f$ resp. $\phi$ act on $\pi^{\mathcal{G}_{M,I}(\mathcal{O}_K)}$ resp. $\Pi^{\mathcal{G}_{M,I}(\mathcal{O}_L)}$, agree: $c_{f,\pi} = c_{\phi,\Pi}$.

Then $f$ and $\phi$ have matching (twisted) orbital integrals.
\end{thm}

\begin{proof} First, because (twisted) characters are locally integrable, cf. \cite{ArthurClozel}, Proposition 2.2, we find that Corollary \ref{BCUnit} implies
\[ \tr( e_{\mathcal{G}_{M,I}(\mathcal{O}_K)} | \pi ) = \tr( (e_{\mathcal{G}_{M,I}(\mathcal{O}_L)},\sigma) | \Pi )\ ,\]
taking $f$ to be the character of $\pi$ and $\delta=1$. This implies that $\pi^{\mathcal{G}_{M,I}(\mathcal{O}_K)}\neq 0$ only if $\Pi^{\mathcal{G}_{M,I}(\mathcal{O}_L)}\neq 0$.

Further, we see that
\[
\tr ( f | \pi ) = c_{f,\pi} \tr( e_{\mathcal{G}_{M,I}(\mathcal{O}_K)} | \pi ) = c_{\phi,\Pi} \tr( (e_{\mathcal{G}_{M,I}(\mathcal{O}_L)},\sigma) | \Pi ) = \tr ( (\phi,\sigma) | \Pi )\ .
\]
We may find a function $f^{\prime}\in C_c^{\infty}(G(K))$ that has matching (twisted) orbital integrals with $\phi$, cf. \cite{ArthurClozel}, Proposition 3.1. This implies that $\tr( (\phi,\sigma) | \Pi ) = \tr( f^{\prime} | \pi )$ by the Weyl integration formula, cf. \cite{ArthurClozel}, p. 36, for the twisted version. Hence $\tr( f - f^{\prime} | \pi ) = 0$ for all tempered irreducible smooth representations $\pi$ of $G(K)$. By Kazhdan's density theorem, Theorem 1 in \cite{Kazhdan}, all regular orbital integrals of $f-f^{\prime}$ vanish. Hence $f$ and $\phi$ have matching regular (twisted) orbital integrals. By \cite{Clozel}, Prop. 7.2, all semi-simple (twisted) orbital integrals of $f$ and $\phi$ match.
\end{proof}

\section{The Geometry of some simple Shimura Varieties}

We briefly recall the construction of integral models for the Shimura varieties considered, as explained in \cite{Haines}, p. 597-600. As in the introduction, fix an imaginary quadratic field $k$ and a central division algebra $D$ over $k$ of dimension $n^2$ and an involution $\ast$ on $D$ of the second kind. Let $\mathbf{G}$ be the algebraic group over $\mathbb{Q}$ representing the following functor on $\mathbb{Q}$-algebras $R$:
\[
\mathbf{G}(R) = \{g\in (D\otimes_\mathbb{Q} R)^{\times} \mid g^{\ast}g\in R^{\times} \}\ .
\]
Then the kernel $\mathbf{G}_0$ of the map $\mathbf{G}\longrightarrow \mathbb{G}_m$ sending $g$ to $g^{\ast}g$ is an inner form of the unitary groups associated to the extension $k$ of $\mathbb{Q}$. We assume that an $\mathbb{R}$-algebra homomorphism
\[
h_0: \mathbb{C}\longrightarrow D_\mathbb{R}
\]
is given such that $h_0(z)^{\ast}=h_0(\overline{z})$ and the involution $x\longmapsto h_0(i)^{-1}x^{\ast}h_0(i)$ is positive. By restriction to $\mathbb{C}^{\times}$, it gives rise to
\[
h: R_{\mathbb{C}/\mathbb{R}} \mathbb{G}_m\longrightarrow \mathbf{G}_{\mathbb{R}}\ .
\]
One can describe the Shimura variety associated to $(\mathbf{G},h^{-1})$ as a moduli space of abelian varieties, of the type considered by Kottwitz in \cite{KottwitzPoints}, cf. also \cite{Haines}.

In the notation of that paper, let $B=D^{\mathrm{op}}$ and $V=D$ as a left $B$-module through right multiplication. Then one gets a natural identification $C=\mathrm{End}_B(V)=D$. We are given $h_0:\mathbb{C}\longrightarrow D_{\mathbb{R}} = C_{\mathbb{R}}$.

Further, there is some $\xi\in D^{\times}$ with $\xi^{\ast}=-\xi$ such that the involution $\ast_B$ on $B=D^{\mathrm{op}}$ given by $x^{\ast_B}=\xi x^{\ast} \xi^{-1}$ is positive: Simply take $\xi$ close to $h_0(i)$ in $D_{\mathbb{R}}$ to achieve the second condition. Then, defining the hermitian form $(\cdot,\cdot): V\times V\longrightarrow \mathbb{Q}$ by
\[
(x,y) = \mathrm{tr}_{k/\mathbb{Q}} \mathrm{tr}_{D/k} (x\xi y^{\ast})\ ,
\]
where $\tr_{D/k}$ is the reduced trace, we have $(bx,y)=(x,b^{\ast_B}y)$ for all $b\in B$. One can further assume that the form $(\cdot,h(i)\cdot)$ is positive definite, where we recall that $h(i)\in D$ acts by left multiplication on $V$.

To summarize, we have a simple $\mathbb{Q}$-algebra $B$ equipped with a positive involution $\ast_B$, a left $B$-module $V$ equipped with a $\ast_B$-hermitian form $(\cdot,\cdot)$, and a homomorphism
\[
h_0:\mathbb{C}\longrightarrow C_{\mathbb{R}} = \mathrm{End}_B(V)\otimes \mathbb{R}
\]
such that the form $(\cdot,h_0(i)\cdot)$ is positive-definite. This is precisely the basic set-up of \cite{KottwitzPoints}.

Using the homomorphism $h_0$, we get a decomposition of
\[
D_\mathbb{C} = V_1\oplus V_2
\]
under the action of $\mathbb{C} \otimes_{\mathbb{R}} \mathbb{C}\cong \mathbb{C}\times \mathbb{C}$. This decomposition is stable under right multiplication of $D_\mathbb{C}=D\otimes_{k,v} \mathbb{C}\oplus D\otimes_{k,v^{\ast}} \mathbb{C}$, where $v,v^{\ast}: k\longrightarrow \mathbb{C}$ are the two embeddings. We make the assumption that as a right $D_{\mathbb{C}}$-module, $V_1$ is isomorphic to
\[
W\oplus (W^{\ast})^{n-1}\ ,
\]
where $W$ and $W^{\ast}$ are the simple right modules for $D\otimes_{k,v} \mathbb{C}$ resp. $D\otimes_{k,v^{\ast}} \mathbb{C}$. Up to exchanging $v$ and $v^{\ast}$, this is equivalent to asking $\mathbf{G}_\mathbb{R}$ to be isomorphic to $\mathbf{GU}(1,n-1)$. In section I.7 of the book \cite{HarrisTaylor}, it is discussed how to achieve this situation.

These assumptions imply that the cocharacter
\[
\mu: \mathbb{G}_m\longrightarrow D_\mathbb{C}=D\otimes_{k,v} \mathbb{C}\oplus D\otimes_{k,v^{\ast}} \mathbb{C}\cong M_n(\mathbb{C})\oplus M_n(\mathbb{C})
\]
associated to $h$ is given by
\[
\mu(z) = \mathrm{diag}(z,1,\ldots,1)\times \mathrm{diag}(1,z,\ldots,z)\ .
\]
Then $\mu$ is already defined as a cocharacter
\[
\mu: \mathbb{G}_m\longrightarrow \mathbf{G}_k\ .
\]

Let $p$ be a prime which splits in $k$ as $p=\mathfrak{p}\overline{\mathfrak{p}}$ and such that $D$ splits at $\mathfrak{p}$ (and $\overline{\mathfrak{p}}$). We fix isomorphisms $k_\mathfrak{p}\cong k_{\overline{\mathfrak{p}}}\cong \mathbb{Q}_p$ and $D_{\mathfrak{p}}\cong M_n(\mathbb{Q}_p)$. This induces an isomorphism
\[
\mathbf{G}_{\mathbb{Q}_p}\cong \GL_n\times \mathbb{G}_m\ .
\]
Under the induced isomorphism
\[
\mathbf{G}_{k_{\mathfrak{p}}}\cong \mathbf{G}_{\mathbb{Q}_p}\cong \GL_n\times \mathbb{G}_m\ ,
\]
the cocharacter $\mu$ takes the form
\[
\mu(z) = \mathrm{diag}(z,1,\ldots,1)\times z\ .
\]

\begin{lem} The representation $r$ of the local $L$-group ${^L}\mathbf{G}_{k_\mathfrak{p}}\cong \GL_n(\mathbb{C})\times \mathbb{C}^{\times}\times W_{k_{\mathfrak{p}}}$ defined by Kottwitz in \cite{KottwitzLambdaAdic}, p. 656, is given by
\[\begin{aligned}
r: \GL_n(\mathbb{C})\times \mathbb{C}^{\times}\times W_{k_{\mathfrak{p}}}&\longrightarrow \GL_n(\mathbb{C}) \\
(g,x,\sigma)&\longmapsto (g^{-1})^t x^{-1}\ .
\end{aligned}\]
\end{lem}

\begin{proof} It clearly satisfies both conditions (a) and (b) in \cite{KottwitzLambdaAdic}, p. 656.
\end{proof}

It remains to fix the integral data. Note that the isomorphism $D_{\mathfrak{p}}\cong M_n(\mathbb{Q}_p)$ gives a natural lattice $M_n(\mathbb{Z}_p)\subset M_n(\mathbb{Q}_p)$, which extends to a unique self-dual lattice $\Lambda\subset V\otimes \mathbb{Q}_p$. The stabilizers of $\Lambda$ give rise to $\mathbb{Z}_{(p)}$-orders $\mathcal{O}_B$, resp. $\mathcal{O}_C$, in $B$, resp. $C$. Locally at $\mathfrak{p}$, we have an isomorphism $\mathcal{O}_{B,\mathfrak{p}}\cong M_n(\mathbb{Z}_p)^{\mathrm{op}}$. Using the involution $\ast_B$, this gives rise to an isomorphism
\[
\mathcal{O}_B\otimes\mathbb{Z}_p\cong M_n(\mathbb{Z}_p)^{\mathrm{op}}\oplus M_n(\mathbb{Z}_p)^{\mathrm{op}}\ ,
\]
such that $\ast_B$ is given by $(X,Y)\longmapsto (Y^t,X^t)$.

Fix a compact open subgroup $K^p$ of $\mathbf{G}(\mathbb{A}_f^p)$. We consider the following moduli scheme.

\begin{prop} Let $\mathfrak{M}$ be the functor from schemes over $\mathcal{O}_{k_{\mathfrak{p}}}\cong \mathbb{Z}_p$ to sets, which maps a locally noetherian scheme $S$ to the set of equivalence classes of abelian varieties $(A,\lambda,\iota,\overline{\eta}^p)$, where
\begin{enumerate}
\item $A$ is projective abelian scheme over $S$ up to prime-to-$p$-isogeny;
\item $\lambda$ is a polarization of $A$ of degree prime to $p$;
\item $\iota: \mathcal{O}_B\longrightarrow \mathrm{End}_S(A)$ such that the Rosati involution restricts to $\ast_B$ on $\mathcal{O}_B$ and the Kottwitz condition
\[
\mathrm{det}_{\mathcal{O}_S}(T-b | \mathrm{Lie} A) = \mathrm{det}_{D/k}(T-b | V_1)
\]
holds for all $b\in \mathcal{O}_B$, where $\mathrm{det}_{D/k}$ is the reduced norm for $D/k$. Note that this equation makes sense because $\mathcal{O}_k$ is mapped to $\mathcal{O}_S$ via $\mathcal{O}_k\longrightarrow \mathcal{O}_{k_{\mathfrak{p}}}\cong \mathbb{Z}_p\longrightarrow \mathcal{O}_S$;
\item $\overline{\eta}^p$ is a level-$K^p$-structure on $A$.
\end{enumerate}
Here, $(A,\lambda,\iota,\overline{\eta}^p)$ and $(A^{\prime},\lambda^{\prime},\iota^{\prime},\overline{\eta}^{p\prime})$ are said to be equivalent if there exists an $\mathcal{O}_B$-linear isogeny $\alpha: A\longrightarrow A^{\prime}$ of degree prime to $p$ such that $\alpha^{\ast}(\lambda^{\prime})=c\lambda$ for some $c\in \mathbb{Z}_{(p)}^{\times}$ and $\alpha_{\ast}(\overline{\eta}^p) = \overline{\eta}^{p\prime}$.

Then $\mathfrak{M}$ is representable by a smooth projective scheme $\mathcal{M}$ over $\mathbb{Z}_p$ of relative dimension $n-1$, if $K^p$ is sufficiently small. Furthermore, if $K=K_pK^p$ for the hyperspecial maximal compact subgroup
\[
K_p=\GL_n(\mathbb{Z}_p)\times \mathbb{Z}_p^{\times}\subset \GL_n(\mathbb{Q}_p)\times \mathbb{Q}_p^{\times}\cong \mathbf{G}(\mathbb{Q}_p)\ ,
\]
then there is an isomorphism
\[
\mathrm{Sh}_K\otimes_k k_{\mathfrak{p}}\cong \mathcal{M}\otimes_{\mathcal{O}_{k_{\mathfrak{p}}}} k_{\mathfrak{p}}\ .
\]
\end{prop}

\begin{proof} This follows from \cite{KottwitzPoints}, once one checks that $\ker^1(\mathbb{Q},\mathbf{G})=1$. By the discussion of Section 7 of \cite{KottwitzPoints}, this is automatic if $n$ is even; for $n$ odd, one has $\ker^1(\mathbb{Q},Z(\mathbf{G}))\buildrel \sim \over \longrightarrow \ker^1(\mathbb{Q},\mathbf{G})$, where $Z(\mathbf{G})$ is the center of $\mathbf{G}$. But, as explained there, for $n$ odd, the center of $Z(\mathbf{G})$ is isomorphic to $\mathrm{Res}_{k/\mathbb{Q}} \mathbb{G}_m$, which has trivial cohomology.
\end{proof}

Associated to $A$ we get a $p$-divisible group $X_A$, which decomposes under the action of
\[
\mathcal{O}_B\otimes \mathbb{Z}_p\cong M_n(\mathbb{Z}_p)^{\mathrm{op}}\times M_n(\mathbb{Z}_p)^{\mathrm{op}}
\]
as
\[
X_A=X_\mathfrak{p}^n\times X_{\overline{\mathfrak{p}}}^n\ .
\]
Furthermore, $X_\mathfrak{p}$ is a $p$-divisible group of dimension $1$ and height $n$ because of the Kottwitz condition and the assumption of signature $(1,n-1)$. Hence the notion of a Drinfeld level structure on $X_\mathfrak{p}$ makes sense.

\begin{prop} Fix $m\geq 1$. Let $\mathfrak{M}_{\Gamma(p^m)}$ be the Galois covering of $\mathfrak{M}$ which parametri\-zes Drinfeld level-$p^m$-structures on $X_\mathfrak{p}$, i.e. sections $P_1,\ldots,P_n: S\longrightarrow X_{\mathfrak{p}}$ such that
\[
X_{\mathfrak{p}}[p^m] = \sum_{(i_1,\ldots,i_n)\in (\mathbb{Z}/p^m\mathbb{Z})^n} [i_1P_1+\ldots +i_nP_n]
\]
as relative Cartier divisors. Then $\mathfrak{M}_{\Gamma(p^m)}$ is representable by a projective scheme $\mathcal{M}_{\Gamma(p^m)}$, if $K^p$ is sufficiently small. Further, if $K=K_pK^p$, where $K_p$ is the compact open subgroup
\[
(1+p^mM_n(\mathbb{Z}_p))\times \mathbb{Z}_p^{\times}\subset \GL_n(\mathbb{Q}_p)\times \mathbb{Q}_p^{\times}\cong \mathbf{G}(\mathbb{Q}_p)\ ,
\]
then
\[
\mathrm{Sh}_K\otimes_k k_{\mathfrak{p}}\cong \mathcal{M}_{\Gamma(p^m)}\otimes_{\mathcal{O}_{k_{\mathfrak{p}}}} k_{\mathfrak{p}}\ .
\]
\end{prop}

\begin{proof} We refer to the discussion in the book of Harris-Taylor, \cite{HarrisTaylor}, sections III.1 and III.4.
\end{proof}

For any direct summand $H$ of $(\mathbb{Z}/p^m\mathbb{Z})^n$, let $\mathcal{M}_{\Gamma(p^m)}^H$ be the reduced subscheme of the closed subscheme of $\mathcal{M}_{\Gamma(p^m)}$ where
\[
\sum_{(i_1,\ldots,i_n)\in H\subset (\mathbb{Z}/p^m\mathbb{Z})^n} [i_1P_1+\ldots+i_nP_n] = |H|[e]\ .
\]

The following theorem is known to the experts, and is easy to deduce from the explicit description of the completed local rings of $\mathcal{M}_{\Gamma(p^m)}$ also used by Yoshida in \cite{Yoshida}.

\begin{thm}\label{SpecialFiberShimuraVariety}\begin{enumerate}
\item[(i)] For all $H$, the scheme $\mathcal{M}_{\Gamma(p^m)}^H$ is regular.
\item[(ii)] The special fibre of $\mathcal{M}_{\Gamma(p^m)}$ is the union of $\mathcal{M}_{\Gamma(p^m)}^H$ over all $H\neq 0$.
\item[(iii)] $\mathcal{M}_{\Gamma(p^m)}^{H_1}\subset \mathcal{M}_{\Gamma(p^m)}^{H_2}$ if and only if $H_2\subset H_1$.
\item[(iv)] This induces a stratification
\[
\mathcal{M}_{\Gamma(p^m)} = \bigcup_H \mathring{\mathcal{M}}_{\Gamma(p^m)}^H
\]
into locally closed strata
\[
\mathring{\mathcal{M}}_{\Gamma(p^m)}^H = \mathcal{M}_{\Gamma(p^m)}^H\setminus \bigcup_{H\subsetneq H^{\prime}} \mathcal{M}_{\Gamma(p^m)}^{H^{\prime}}\ .
\]
\end{enumerate}
\end{thm}

\begin{proof} We may pass to the formal complection at a geometric point $x$ of the special fibre. Then we get a deformation space of the $p$-divisible group $X_\mathfrak{p}$ of dimension $1$ and height $n$: Deforming $(A,\iota,\lambda,\overline{\eta}^p,P_1,\ldots,P_n)$ is the same as deforming $(A[p^{\infty}],\iota,\lambda,P_1,\ldots,P_n)$ by the Serre-Tate theorem and rigidity of level structures away from $p$. As $\lambda: A[\mathfrak{p}^{\infty}]\cong A[\overline{\mathfrak{p}}^{\infty}]^{\vee}$, this is the same as deforming $(A[\mathfrak{p}^{\infty}],\iota|_{M_n(\mathbb{Z}_p)^{\mathrm{op}}},P_1,\ldots,P_n)$. But this reduces to deforming $(X_\mathfrak{p},P_1,\ldots,P_n)$, as $A[\mathfrak{p}^{\infty}]\cong X_\mathfrak{p}^n$ through the action of $\iota|_{M_n(\mathbb{Z}_p)^{\mathrm{op}}}$.

Let $X_\mathfrak{p} = X_{\mathrm{inf}}\times X_{\mathrm{et}}$ be the decomposition into \'{e}tale and infinitesimal part, where $X_{\mathrm{inf}}$ has height $k$. Applying some element of $\GL_n(\mathbb{Z}/p^m\mathbb{Z})$ if necessary, we can assume that $P_1=\ldots=P_k=e$ and $P_{k+1},\ldots,P_n$ generate $X_{\mathrm{et}}$. Then the deformation space of $(X_\mathfrak{p},P_1,\ldots,P_n)$ maps to the deformation space of $(X_{\mathrm{inf}},P_1,\ldots,P_k)$ and this map is formally smooth, as shown in \cite{HarrisTaylor}, p.80. If $x$ lies in $\mathcal{M}_{\Gamma(p^n)}^H$, then necessarily $H\subset (\mathbb{Z}/p^m\mathbb{Z})^k \oplus 0^{n-k}$. Further, these strata $\mathcal{M}_{\Gamma(p^m)}^H$ are pullbacks of the corresponding strata for the deformation space of $(X_{\mathrm{inf}},P_1,\ldots,P_k)$, as their equations only involve the deformation of $X_{\mathrm{inf}}$ and $P_1,\ldots,P_k$. We are reduced to the case of the deformation of a $1$-dimensional formal group $X_{\mathrm{inf}}$ with Drinfeld-level-structure $P_1,\ldots,P_k$.

By the results of Drinfeld, \cite{Drinfeld}, Proposition 4.3, the deformation ring $R$ of
\[
(X_{\mathrm{inf}},P_1,\ldots,P_k)
\]
is a complete regular local ring with parameters $X_1,\ldots,X_k$. There is a group structure $+_{\Sigma}$ on the maximal ideal $\mathfrak{m}$ of $R$ given by the universal deformation $\Sigma$ of the formal group law of $X_{\mathrm{inf}}$, after the choice of a formal parameter on the universal deformation of $X_{\mathrm{inf}}$. We may assume that $H=(\mathbb{Z}/p^m\mathbb{Z})^j\oplus 0^{n-j}\subset (\mathbb{Z}/p^m\mathbb{Z})^n$ for some $j\leq k$. The condition
\[
|H|[e] = \sum_{i_1,\ldots,i_j\in \mathbb{Z}/p^m\mathbb{Z}} [i_1P_1+\ldots +i_jP_j]
\]
is equivalent to
\[
T^{p^{mj}} = (\mathrm{unit}) \prod_{i_1,\ldots,i_j\in \mathbb{Z}/p^m\mathbb{Z}} (T-([i_1](X_1)+_{\Sigma}\ldots+_{\Sigma}[i_j](X_j)))
\]
as power series in the formal variable $T$. It follows that all symmetric polynomials in the variables $x_{i_1,\ldots,i_j} = [i_1](X_1)+_{\Sigma}\ldots+_{\Sigma}[i_j](X_j)$ vanish, hence on the reduced subscheme all $x_{i_1,\ldots,i_j}$ vanish. This implies $X_1=\ldots=X_j=0$ on the reduced closed subscheme, and conversely if $X_1=\ldots=X_j=0$, then $|H|[e] = \sum [i_1P_1+\ldots +i_jP_j]$. Because $X_1,\ldots,X_k$ form a regular sequence, the closed subscheme given by $X_1=\ldots=X_j=0$ is regular and in particular reduced. This describes the closed subscheme corresponding to $H$ locally and claims (i)-(iii) are clear.

For claim (iv), it is enough to check that for any point $x\in \mathcal{M}_{\Gamma(p^m)}$, there is a unique maximal $H$ such that $x\in \mathcal{M}_{\Gamma(p^m)}^H$. This is also clear from our description.
\end{proof}

\section{Description of isogeny classes}

We recall some results from \cite{KottwitzPoints} that we will need. In that paper, Kottwitz associates to any point $x\in \mathcal{M}(\mathbb{F}_{p^r})$ a triple $(\gamma_0,\gamma,\delta)$ consisting of
\begin{enumerate}
\item a semisimple element $\gamma_0\in \mathbf{G}(\mathbb{Q})$, elliptic in $\mathbf{G}(\mathbb{R})$ and well-defined up to $\mathbf{G}(\overline{\mathbb{Q}})$-conjugation;
\item an element $\gamma\in \mathbf{G}(\mathbb{A}_f^p)$, well-defined up to conjugacy;
\item an element $\delta\in \mathbf{G}(\mathbb{Q}_{p^r})$, well-defined up to $\sigma$-conjugacy;
\end{enumerate}
satisfying certain compatibilities, e.g. the image of $\gamma_0$ in $\mathbf{G}(\mathbb{Q}_{\ell})$ is stably conjugate to the $\ell$-adic component of $\gamma$ and the image of $\gamma_0$ in $\mathbf{G}(\mathbb{Q}_p)$ is stably conjugate to $N\delta$. Roughly, these parametrize $\mathbb{F}_{p^r}$-isogeny classes in $\mathcal{M}(\mathbb{F}_{p^r})$, cf. \cite{KottwitzPoints} for precise statements.

Let us briefly recall how $\delta$ is defined, as this is what is used below, cf. \cite{KottwitzPoints}, p. 419. Let $x$ correspond to $(A,\iota,\lambda,\overline{\eta}^p)$. The dual of the rational crystalline cohomology
\[
H_p=\mathrm{Hom}(H^1_{\mathrm{cris}}(A/\mathbb{Z}_{p^r}), \mathbb{Q}_{p^r})
\]
has an action of $B$, a semilinear Frobenius $F$ and a principal polarization $\lambda$. As shown in \cite{KottwitzPoints}, p.430, there is an isomorphism of $B$-modules $H_p\cong V\otimes_{\mathbb{Q}} \mathbb{Q}_{p^r}$, preserving the hermitian forms up to a scalar. Choosing such an isomorphism, $F$ is mapped to an endomorphism $\delta\sigma$ for some skew-hermitian $B$-linear automorphism $\delta$ of $V\otimes_{\mathbb{Q}} \mathbb{Q}_{p^r}$, i.e. $\delta\in \mathbf{G}(\mathbb{Q}_{p^r})$, which is well-defined up to $\sigma$-conjugacy.

In the following lemma, this is computed more directly.

\begin{lem}\label{CompDelta} Let $x\in \mathcal{M}(\mathbb{F}_{p^r})$. One has the associated $p$-divisible group $X_{\mathfrak{p},x}$ and hence a (contravariant) rational Dieudonn\'{e} module $N_x$, abstractly isomorphic to $\mathbb{Q}_{p^r}^n$. Fixing such an isomorphism, the Frobenius operator $F$ takes the form $\delta_0\sigma$, and under the isomorphism
\[
\mathbf{G}_{\mathbb{Q}_p}\cong \GL_n\times \mathbb{G}_m\ ,
\]
the element $\delta$ is sent to $((\delta_0^{-1})^t,p^{-1})$. Moreover, $\delta_0$, resp. $\delta$, is uniquely determined as an element of $\GL_n(\mathbb{Q}_{p^r})\cap M_n(\mathbb{Z}_{p^r})$, resp. $\GL_n(\mathbb{Q}_{p^r})\times \mathbb{Q}_{p^r}^{\times}$, up to $\GL_n(\mathbb{Z}_{p^r})$-$\sigma$-conjugation, resp. $\mathbf{G}(\mathbb{Z}_{p^r})$-$\sigma$-conjugation.
\end{lem}

\begin{rem} This stronger normalization of $\delta$ will be essential.
\end{rem}

\begin{proof} Fix an isomorphism $\iota: N_x\longrightarrow \mathbb{Q}_{p^r}^n$. Recall that the $p$-divisible group $X_A$ of the universal abelian variety $A$ decomposes as
\[
X_A = X_{\mathfrak{p}}^n \times X_{\overline{\mathfrak{p}}}^n\ ,
\]
equivariant for action of $\mathcal{O}_B\otimes \mathbb{Z}_p\cong M_n(\mathbb{Z}_p)^{\mathrm{op}}\times M_n(\mathbb{Z}_p)^{\mathrm{op}}$, and the polarization $\lambda$ gives an isomorphism
\[
\lambda: X_{\mathfrak{p}}^n\cong (X_{\overline{\mathfrak{p}}}^n)^{\vee}\ .
\]
In particular,
\[
H_p\cong \mathrm{Hom}(N_x^n\oplus (N_x^{\vee})^n,\mathbb{Q}_{p^r})\ ,
\]
compatible with $\mathcal{O}_B$-action and polarization. Hence
\[
H_p\cong \mathrm{Hom}(N_x^n\oplus (N_x^{\vee})^n,\mathbb{Q}_{p^r})\cong M_n(\mathbb{Q}_{p^r})\oplus M_n(\mathbb{Q}_{p^r})\cong V\otimes_{\mathbb{Q}} \mathbb{Q}_{p^r}
\]
is compatible with the action of $\mathcal{O}_B$ and the hermitian forms, up to the scalar $p$. The scalar $p$ enters, because by the definition of $N_x^{\vee}$, there is a perfect pairing of Dieudonn\'{e} modules
\[
N_x\times N_x^{\vee}\longrightarrow \mathbb{Q}_{p^r}(1)\ .
\]
This implies the first part of the lemma.

One gets the stronger normalization of $\delta_0$ by picking an isomorphism of the integral Dieu\-donn\'{e} module of $X_{\mathfrak{p},x}$ with $\mathbb{Z}_{p^r}^n$.
\end{proof}

\section{Calculation of the nearby cycles}\label{SectionNearbyCycles}

In this section we consider the following general situation. Let $\mathcal{O}$ be the ring of integers in a local field $K$. Let $X/ \mathcal{O}$ be a separated scheme of finite type. Let $X_{\eta^{\mathrm{ur}}}$ be the base-change of $X$ to the maximal unramified extension $K^{\mathrm{ur}}$ of $K$ and let $X_{\mathcal{O}^{\mathrm{ur}}}$ be the base-change to the ring of integers in $K^{\mathrm{ur}}$. Also let $X_{\overline{s}}$ be the geometric special fibre of $X$. Then we have $\iota: X_{\overline{s}}\longrightarrow X_{\mathcal{O}^{\mathrm{ur}}}$ and $j: X_{\eta^{\mathrm{ur}}}\longrightarrow X_{\mathcal{O}^{\mathrm{ur}}}$. Let $R_{I_K}$ be the derived functor of taking invariants under the inertia group $I_K$.

\begin{lem}\label{InertiaNearby} For any sheaf $\mathcal{F}$ on the generic fibre $X_{\eta}$,
\[ R_{I_K}(R\psi \mathcal{F})=\iota^{\ast}Rj_{\ast} \mathcal{F}_{\eta^{\mathrm{ur}}}. \]
\end{lem}

\begin{proof} Both sides are the derived functors of the same functor, cf. \cite{Scholze}, Lemma 8.1.
\end{proof}

We now make further assumptions. We assume that $X$ is regular and flat of relative dimension $n$ over $\mathcal{O}$ and that the special fibre $X_s$ possesses a stratification $X_s = \bigcup_i \mathring{Z}_i$ into locally closed strata $\mathring{Z}_i$ whose closures $Z_i$ are regular. We call the $Z_i$ special fibre components. Let $c(Z)$ be the codimension of $Z$ in $X$, for any special fibre component $Z$ in $X$.

To any special fibre component $Z$, we associate a $\overline{\mathbb{Q}}_{\ell}$-vector space $W_Z$ together with maps $W_Z\longrightarrow W_{Z^{\prime}}$ whenever $c(Z)=c(Z^{\prime})+1$, by induction on $c(Z)$. If $c(Z)=1$, we simply set $W_Z=\overline{\mathbb{Q}}_{\ell}$. In general, we define $W_Z$ to be the kernel of the map
\[
\bigoplus_{\substack{Z\subset Z^{\prime}\\ c(Z^{\prime})=c(Z)-1}} W_{Z^{\prime}} \longrightarrow \bigoplus_{\substack{Z\subset Z^{\prime}\\ c(Z^{\prime})=c(Z)-2}} W_{Z^{\prime}}\ ,
\]
where in the case $c(Z)=2$, the target has to be replaced by $\overline{\mathbb{Q}}_{\ell}$. We make the following assumption:
\vspace{0.1in}

($\ast$) The sequence
\[
0\longrightarrow W_Z\longrightarrow \ldots \longrightarrow \bigoplus_{\substack{Z\subset Z^{\prime}\\ c(Z^{\prime})=i}} W_{Z^{\prime}} \longrightarrow \bigoplus_{\substack{Z\subset Z^{\prime}\\ c(Z^{\prime})=i-1}} W_{Z^{\prime}} \longrightarrow \ldots \longrightarrow \overline{\mathbb{Q}}_{\ell}\longrightarrow 0
\]
is exact for all special fibre components $Z$.

\begin{rem} By definition of $W_Z$, this sequence is exact at the first step. It is not clear to the author whether the condition ($\ast$) is automatic in general. Using a combinatorial result of Folkman on the homology of geometric lattices, \cite{Folkman}, one can check that it is fulfilled if the reduced intersection of any two special fibre components is again a special fibre component. Note that in the general case, this intersection will only be a union of special fibre components.
\end{rem}

Finally, let $x\in X_s(\mathbb{F}_q)$ be a point of the special fibre, with geometric point $\overline{x}\in X_s(\overline{\mathbb{F}}_q)$ over $x$.

\begin{thm}\label{CalcInertiaNearby} Assume that ($\ast$) holds. Then for all $k$, there are canonical isomorphisms
\[
(\iota^{\ast}R^kj_{\ast} \overline{\mathbb{Q}}_{\ell})_{\overline{x}}\cong \bigoplus_{\substack{Z, c(Z)=k\\ x\in Z}} W_Z(-k)\ .
\]
\end{thm}

\begin{proof} Let $b_Z: Z\longrightarrow X$ be the closed embeddings, for all special fibre components $Z$. Let $I^{\bullet}$ be an injective resolution of the constant sheaf $\overline{\mathbb{Q}}_{\ell}$ on $X$. Consider the following complex
\[\begin{aligned}
0\longrightarrow\ldots &\longrightarrow \bigoplus_{Z, c(Z)=i} W_Z\otimes b_{Z\ast} b_Z^{!} I^{\bullet} \longrightarrow \ldots \\
\ldots&\longrightarrow \bigoplus_{Z, c(Z)=1} W_Z\otimes b_{Z\ast} b_Z^{!} I^{\bullet} \longrightarrow \iota^{\ast} I^{\bullet} \longrightarrow \iota^{\ast} j_{\ast} j^{\ast} I^{\bullet}\longrightarrow 0\ .
\end{aligned}\]

\begin{prop} The hypercohomology of this complex vanishes.
\end{prop}

\begin{proof} We prove the proposition for any complex of injective sheaves $I^{\bullet}$. This reduces the problem to doing it for a single injective sheaf $I$. Recall that for any injective sheaf $I$ on a scheme $X=U\cup Z$ with $j: U\longrightarrow X$ an open and $i: Z\longrightarrow X$ a closed embedding, one gets a decomposition $I=j_{\ast} I_U\oplus i_{\ast} I_Z$ for certain injective sheaves $I_U$ and $I_Z$ on $U$, resp. $Z$. In our situation, we get a decomposition of $I$ as
\[ I = \bigoplus_Z f_{Z\ast} I_Z\ , \]
where $I_Z$ is an injective sheaf on $\mathring{Z}$ and where $f_Z: \mathring{Z}\longrightarrow X$ is the natural locally closed embedding. But that the complex is exact for
\[
I = f_{Z\ast} I_Z
\]
is a direct consequence of the condition ($\ast$) on the vector spaces $W_Z$.
\end{proof}

As in \cite{Scholze}, proof of Theorem 8.2, this implies the Theorem, using cohomological purity, as proved by Thomason, \cite{Thomason}.
\end{proof}

As a corollary of Theorem \ref{CalcInertiaNearby}, we can compute the semisimple trace of Frobenius on the nearby cycles in our situation. We start by analyzing the combinatorics of the situation that will arise. Recall from Theorem \ref{SpecialFiberShimuraVariety} that the special fibre components are parametrized by (nontrivial) direct summands of $(\mathbb{Z}/p^m\mathbb{Z})^n$. The next lemma shows that the condition ($\ast$) holds in this case. Note that if a group $G$ acts on $X$ preserving the stratification of the special fibre, then all vector spaces $W_Z$ acquire an action of the stabilizer of $Z$ in $G$. In our case
\[
\mathcal{M}_{\Gamma(p^m)}\otimes \mathbb{F}_p = \bigcup_{H\neq 0} \mathring{\mathcal{M}}_{\Gamma(p^m)}^H\ ,
\]
and the stabilizer of $\mathring{\mathcal{M}}_{\Gamma(p^m)}^H$ is $\GL(H)\cong \GL_k(\mathbb{Z}/p^m\mathbb{Z})$, if $H\cong (\mathbb{Z}/p^m\mathbb{Z})^k$.

\begin{lem}\label{CombSteinberg} Let the rank of $H$ be $k$. Then the representation $W_H$ of $\GL(H)$ on the $\overline{\mathbb{Q}}_{\ell}$-vector space of functions $f: \{0=H_0\subsetneq H_1\subsetneq \cdots \subsetneq H_k=H\}\longrightarrow \overline{\mathbb{Q}}_{\ell}$ on complete flags of direct summands of $H$ satisfying
\[
\sum_{\substack{H_i^{\prime}\ \mathrm{direct}\ \mathrm{summand}\\ H_0\subsetneq \cdots \subsetneq H_{i-1}\subsetneq H_i^{\prime}\subsetneq H_{i+1}\subsetneq \cdots \subsetneq H_k=H}} f(H_0\subsetneq \cdots \subsetneq H_i^{\prime}\subsetneq \cdots \subsetneq H_k) = 0
\]
for all $i=1,\ldots,k-1$ is isomorphic to the Steinberg representation $\mathrm{St}_H$ of $\GL(H)$ given by
\[
\mathrm{St}_H = \mathrm{ker}\left(\mathrm{Ind}_{B}^{\GL(H)} 1 \longrightarrow \bigoplus_{B\subsetneq P} \mathrm{Ind}_{P}^{\GL(H)} 1\right)\ ,
\]
where $B\subset \GL(H)$ is a Borel subgroup, $P\subset \GL(H)$ runs through parabolic subgroups and $1$ denotes the trivial representation. Moreover, if $H^{\prime}\subset H$ is a direct summand of rank $k-1$, then there are natural transition maps $W_H\longrightarrow W_{H^{\prime}}$ defined by
\[
f^{\prime}(H_0\subsetneq\cdots\subsetneq H_{k-1}=H^{\prime}) = f(H_0\subsetneq\cdots\subsetneq H_{k-1}\subsetneq H_k=H)\ .
\]

Finally, the complex
\[
0\longrightarrow W_H\longrightarrow \ldots \longrightarrow \bigoplus_{\substack{H^{\prime}\subset H\\ \mathrm{rank}\ H^{\prime}=i}} W_{H^{\prime}} \longrightarrow \bigoplus_{\substack{H^{\prime}\subset H\\ \mathrm{rank}\ H^{\prime}=i-1}} W_{H^{\prime}} \longrightarrow \ldots \longrightarrow \overline{\mathbb{Q}}_{\ell}\longrightarrow 0
\]
is exact.
\end{lem}

\begin{proof} Under the standard identifications of the homogeneous spaces with spaces of flags, the first assertion is immediate.

For the last assertion, assume that $H=(\mathbb{Z}/p^m)^k$, so that $\GL(H)=\GL_k(\mathbb{Z}/p^m)$. Let $P_i\subset\GL_k$ be the standard parabolic with Levi $\GL_i\times \GL_{k-i}$ for $i=1,\ldots,k-1$. Let $\mathrm{St}_i$ be the Steinberg representation of $\GL_i(\mathbb{Z}/p^m)$. Then the complex gets identified with
\[\begin{aligned}
0\longrightarrow \mathrm{St}_k&\longrightarrow \mathrm{Ind}_{P_1(\mathbb{Z}/p^m)}^{\GL_k(\mathbb{Z}/p^m)} 1\otimes \mathrm{St}_{k-1}\longrightarrow \cdots\longrightarrow \mathrm{Ind}_{P_i(\mathbb{Z}/p^m)}^{\GL_k(\mathbb{Z}/p^m)} 1\otimes \mathrm{St}_{i}\longrightarrow\cdots\\
 \cdots &\longrightarrow \mathrm{Ind}_{P_{k-1}(\mathbb{Z}/p^m)}^{\GL_k(\mathbb{Z}/p^m)} 1\otimes 1\longrightarrow 1\longrightarrow 0\ .
\end{aligned}\]
This is the complex of $1+p^mM_k(\mathbb{Z}_p)$-invariants in a corresponding complex
\[\begin{aligned}
0\longrightarrow \mathrm{St}_k&\longrightarrow \mathrm{Ind}_{P_1(\mathbb{Q}_p)}^{\GL_k(\mathbb{Q}_p)} 1\otimes \mathrm{St}_{k-1}\longrightarrow \cdots\longrightarrow \mathrm{Ind}_{P_i(\mathbb{Q}_p)}^{\GL_k(\mathbb{Q}_p)} 1\otimes \mathrm{St}_{i}\longrightarrow\cdots\\
 \cdots &\longrightarrow \mathrm{Ind}_{P_{k-1}(\mathbb{Q}_p)}^{\GL_k(\mathbb{Q}_p)} 1\otimes 1\longrightarrow 1\longrightarrow 0\ .
\end{aligned}\]
of $\GL_k(\mathbb{Q}_p)$-representations. But note that $\mathrm{Ind}_{P_i(\mathbb{Q}_p)}^{\GL_k(\mathbb{Q}_p)} 1\otimes \mathrm{St}_{i}$ is an extension of two irreducible representations $\pi_{i-1}$ and $\pi_i$ of $\GL_n(\mathbb{Q}_p)$, cf. e.g. Lemma I.3.2 of \cite{HarrisTaylor}. Let $d^i$ denote the $i$-th differential of this complex. As all differentials of the complex are nonzero, one sees by induction that $\im d^{i+1} = \pi_i = \ker d^i$, whence the complex is exact.
\end{proof}

\begin{cor}\label{NearbyCyclesShimuraVar} The condition ($\ast$) holds for the scheme $\mathcal{M}_{\Gamma(p^m)}$, and $W_Z=W_H$ if $Z=\mathcal{M}_{\Gamma(p^m)}^H$. In particular, let $x\in \mathcal{M}_{\Gamma(p^m)}(\mathbb{F}_q)$ be a point such that the infinitesimal part of the associated $p$-divisible group $X_{\mathfrak{p},x}$ has height $k$, i.e.
\[
x\in \mathring{\mathcal{M}}_{\Gamma(p^m)}^H(\mathbb{F}_q)
\]
for some direct summand $H\subset (\mathbb{Z}/p^m\mathbb{Z})^n$ of rank $k$. Fix an isomorphism $H\cong (\mathbb{Z}/p^m\mathbb{Z})^k$. Then
\[
(R_{I_F}^j R\psi_{\mathcal{M}_{\Gamma(p^m)}}\bar{\mathbb{Q}}_{\ell})_{\overline{x}} \cong \left(\mathrm{Ind}_{P_{j,k}(\mathbb{Z}/p^m\mathbb{Z})}^{\GL_k(\mathbb{Z}/p^m\mathbb{Z})} \mathrm{St}_j\otimes 1\right) (-j)\ ,
\]
where $P_{j,k}\subset \GL_k$ is the standard parabolic with Levi $\GL_j\times \GL_{k-j}$.\hfill $\Box$
\end{cor}

Fix a positive integer $r$ and let $1\leq k\leq n$. We define certain (virtual) representations. First, we define the virtual representation
\[
I_k^0 = \frac{1}{1-p^r}\sum_{j=0}^k (-1)^j p^{rj} \mathrm{Ind}_{P_{j,k}(\mathbb{Z}_p)}^{\GL_k(\mathbb{Z}_p)} \mathrm{St}_j\otimes 1
\]
of $\GL_k(\mathbb{Z}_p)$, where $\mathrm{St}_j$ is the Steinberg representation of $\GL_j(\mathbb{Z}_p)$. Note that $I_k^0$ is self-dual. In the situation of the corollary, it follows from Lemma 7.7 of \cite{Scholze} that for all $g\in \GL_k(\mathbb{Z}/p^m\mathbb{Z})$ we have
\[
\mathrm{tr}^{\mathrm{ss}}(\Phi_{p^r}\times g | (R\psi_{\mathcal{M}_{\Gamma(p^m)}}\bar{\mathbb{Q}}_{\ell})_{\overline{x}} ) = \tr(ge_{\Gamma(p^m)}|I_k^0)\ ,
\]
where $e_{\Gamma(p^m)}$ is the idempotent associated to $\Gamma(p^m)=1+p^mM_n(\mathbb{Z}_p)$. Moreover, the long exact sequence in Lemma \ref{CombSteinberg} shows that $I_k^0$ can be rewritten as a $\mathbb{Z}$-linear combination of representations:
\[
I_k^0 = \sum_{j=0}^{k-1} p^{rj} \sum_{i=0}^j (-1)^i \mathrm{Ind}_{P_{i,k}(\mathbb{Z}_p)}^{\GL_k(\mathbb{Z}_p)} \mathrm{St}_i\otimes 1\ .
\]
We remark without proof that the virtual representation
\[
\sum_{i=0}^j (-1)^i \mathrm{Ind}_{P_{i,k}(\mathbb{Q}_p)}^{\GL_k(\mathbb{Q}_p)} \mathrm{St}_i\otimes 1
\]
is up to the sign $(-1)^j$ the irreducible subquotient of $\mathrm{Ind}_{B(\mathbb{Q}_p)}^{\GL_k(\mathbb{Q}_p)} 1$ corresponding to $\{1,\ldots,j+1\}\subset \{1,\ldots,k\}$ in the standard enumeration of the irreducible subquotients. These representations also show up in the cohomology of Drinfeld's upper half space, cf. \cite{SchneiderStuhler}.

Consider the standard parabolic $P_k\subset \GL_n$ with Levi subgroup $\GL_k\times \GL_{n-k}$ and define the virtual representation
\[
I_k = \mathrm{Ind}_{P_k(\mathbb{Z}_{p})}^{\GL_n(\mathbb{Z}_{p})} I_k^0\otimes C_c^{\infty}(\GL_{n-k}(\mathbb{Z}_p))
\]
which carries a left action of $\GL_n(\mathbb{Z}_p)$ and a left action of $g\in \GL_{n-k}(\mathbb{Z}_p)$ through multiplication by $g^{-1}$ on the right.

Next, we use these representations to define the test functions $\phi_h$ whose twisted orbital integrals will appear in the formula for the Lefschetz number. We need to introduce some terminology.

\begin{definition} An element $\delta_0\in \GL_n(\mathbb{Q}_{p^r})\cap M_n(\mathbb{Z}_{p^r})$ is said to be of height $k$ if $v_p(\delta_0)=1$ and $\delta_0$ is $\GL_n(\mathbb{Z}_{p^r})$-$\sigma$-conjugate to an element $(\delta_1,\delta_2)\in (\GL_k(\mathbb{Q}_{p^r})\cap M_k(\mathbb{Z}_{p^r}))\times \GL_{n-k}(\mathbb{Z}_{p^r})$ such that $N\delta_1$ is elliptic.
\end{definition}

It is clear that if $\delta_0$ is of height $k$, then
\[
\delta_0\in \GL_n(\mathbb{Z}_{p^r})\mathrm{diag}(p,1,\ldots,1)\GL_n(\mathbb{Z}_{p^r})\ .
\]
On the other hand, it is easy to see that one-dimensional $p$-divisible groups $X$ of height $n$ over $\mathbb{F}_{p^r}$ are in bijection with $\GL_n(\mathbb{Z}_{p^r})$-$\sigma$-conjugacy classes
\[
\delta_0\in \GL_n(\mathbb{Z}_{p^r})\mathrm{diag}(p,1,\ldots,1)\GL_n(\mathbb{Z}_{p^r})\ ,
\]
and under this bijection, $\delta_0$ is of height $k$ if and only if the infinitesimal part of $X$ has height $k$. In particular, this shows that any $\delta_0\in \GL_n(\mathbb{Z}_{p^r})\mathrm{diag}(p,1,\ldots,1)\GL_n(\mathbb{Z}_{p^r})$ is of height $k$ for a unique $k=1,\ldots,n$.

\begin{definition}\label{DefPhih} Let $h\in C_c^{\infty}(\GL_n(\mathbb{Z}_p))$. Let $\phi_h\in C_c^{\infty}(\GL_n(\mathbb{Q}_{p^r}))$ be the unique function with support in $\GL_n(\mathbb{Z}_{p^r})\mathrm{diag}(p,1,\ldots,1)\GL_n(\mathbb{Z}_{p^r})$ and invariant under $\GL_n(\mathbb{Z}_{p^r})$-$\sigma$-conjugation, such that if $\delta_0 = (\delta_1,\delta_2)\in (\GL_k(\mathbb{Q}_{p^r})\cap M_k(\mathbb{Z}_{p^r}))\times \GL_{n-k}(\mathbb{Z}_{p^r})$ is of height $k$, then
\[
\phi_h(\delta_0) = \tr( h \times N\delta_2 | I_k )\ .
\]
\end{definition}

This function will take the place of the characteristic function, often called $\phi_r$, of
\[
\GL_n(\mathbb{Z}_{p^r})\mathrm{diag}(p,1,\ldots,1)\GL_n(\mathbb{Z}_{p^r})
\]
whose twisted orbital integrals intervene in the proof of the Theorem D for the case that $K_p$ is a maximal compact subgroup.

Recall that we associated an element $\delta_0\in \mathrm{GL}_n(\mathbb{Q}_{p^r})\cap M_n(\mathbb{Z}_{p^r})$ to any point $x\in \mathcal{M}(\mathbb{F}_{p^r})$, well-defined up to $\GL_n(\mathbb{Z}_{p^r})$-$\sigma$-conjugation. Fix a geometric point $\overline{x}$ over $x$. We have the covering
\[
\pi_m: \mathcal{M}_{\Gamma(p^m)}\longrightarrow \mathcal{M}
\]
and the sheaf $\mathcal{F}_m = \pi_{m\eta\ast} \mathbb{Q}_{\ell}$ on the generic fibre of $\mathcal{M}_{\Gamma(p^m)}$. Define the vector space
\[
(R\psi \mathcal{F}_{\infty})_{\overline{x}} = \lim_{\longrightarrow} (R\psi \mathcal{F}_m)_{\overline{x}}\ .
\]
It carries a natural smooth admissible action of $\mathrm{GL}_n(\mathbb{Z}_p)$ and a commuting continuous action of $\mathrm{Gal}(\overline{\mathbb{Q}}_{p^r}/\mathbb{Q}_{p^r})$.

Finally, fix $h\in C_c^{\infty}(\GL_n(\mathbb{Z}_p))$ and define a new function $h^{\vee}\in C_c^{\infty}(\GL_n(\mathbb{Z}_p))$ by $h^{\vee}(g)=h((g^{-1})^t)$.

\begin{thm}\label{CalcSemisimpleTrace} In this situation,
\[
\mathrm{tr}^{\mathrm{ss}}(\Phi_{p^r}\times h^{\vee} | (R\psi\mathcal{F}_{\infty})_{\overline{x}} ) = \phi_h(\delta_0)\ .
\]
\end{thm}

\begin{proof} We can assume that $h$ equals $ge_{\Gamma(p^m)}$ for some $g\in \GL_n(\mathbb{Z}/p^m\mathbb{Z})$.  Let us describe the fiber $\pi_m^{-1}(\overline{x})$ over $\overline{x}\in \mathcal{M}(\overline{\mathbb{F}}_{p^r})$. The Drinfeld level-$p^m$-structures are parametrized by $n$-tuples
\[
(P_1,\ldots,P_n)\in X_{\mathrm{et},\overline{x}}[p^m](\overline{\mathbb{F}}_{p^r})
\]
generating $X_{\mathrm{et},\overline{x}}[p^m](\overline{\mathbb{F}}_{p^r})$. The action of $\Phi_{p^r}$ on $\pi_m^{-1}(\overline{x})$ is given by the action of $\Phi_{p^r}$ on $X_{\mathrm{et},\overline{x}}[p^m](\overline{\mathbb{F}}_{p^r})$, which is given by right multiplication with $(N\delta_2)^{-1}$.

Further, Lemma 7.7 of \cite{Scholze} says that
\[
\mathrm{tr}^{\mathrm{ss}}(\Phi_{p^r}\times (g^{-1})^t | (R\psi\mathcal{F}_m)_{\overline{x}} ) = \frac 1{1-p^r}\mathrm{tr}(\Phi_{p^r}\times (g^{-1})^t | R_{I_{\mathbb{Q}_p}} (R\psi\mathcal{F}_m)_{\overline{x}} )\ .
\]

But one can rewrite
\[\begin{aligned}
R_{I_{\mathbb{Q}_p}} R\psi_{\mathcal{M}} \mathcal{F}_m = R_{I_{\mathbb{Q}_p}} R\psi_{\mathcal{M}} \pi_{m\eta\ast} \overline{\mathbb{Q}}_{\ell} &= \pi_{m\overline{s}\ast} R_{I_{\mathbb{Q}_p}} R\psi_{\mathcal{M}_{\Gamma(p^m)}} \overline{\mathbb{Q}}_{\ell} \\
& = \pi_{m\overline{s}\ast} \iota_{\mathcal{M}_{\Gamma(p^m)}}^{\ast} Rj_{\mathcal{M}_{\Gamma(p^m)}\ast} \overline{\mathbb{Q}}_{\ell}\ ,
\end{aligned}\]
because $\pi_m$ is finite. Here subscripts for $R\psi$ indicate with respect to which scheme the nearby cycles are taken, and $\iota_{\mathcal{M}_{\Gamma(p^m)}}$ and $j_{\mathcal{M}_{\Gamma(p^m)}}$ are as defined before Lemma \ref{InertiaNearby}, for the scheme $\mathcal{M}_{\Gamma(p^m)}$.

Now the theorem follows from Corollary \ref{NearbyCyclesShimuraVar} and the definition of $I_k$.
\end{proof}

\section{Orbital integrals for $\GL_n$}

Fix an integer $r\geq 1$ throughout this section. First, we construct the function $f_{n,p}$ which will turn out to have the correct orbital integrals.

\begin{lem}\label{FunctionExistsGLn} There is a function $f_{n,p}$ of the Bernstein center for $\GL_n(\mathbb{Q}_p)$ such that for all irreducible smooth representations $\pi$ of $\GL_n(\mathbb{Q}_p)$, $f_{n,p}$ acts by the scalar
\[ p^{\frac {n-1}2 r} \mathrm{tr}^{\mathrm{ss}}(\Phi_p^r|\sigma_{\pi})\ , \]
where $\sigma_{\pi}$ is the representation of the Weil group $W_{\mathbb{Q}_p}$ of $\mathbb{Q}_p$ with values in $\overline{\mathbb{Q}}_{\ell}$ associated to $\pi$ by the Local Langlands Correspondence.
\end{lem}

\begin{proof} The proof is identical to the proof of Lemma 9.1 in \cite{Scholze}.
\end{proof}

\begin{rem}\label{DefSemisimpleTrace} The definition as given needs the existence of the Local Langlands Correspondence. However, one can give a direct definition of $f_{n,p}$, because it is easy to evaluate $\mathrm{tr}^{\mathrm{ss}}(\Phi_p^r|\sigma_{\pi})$: If $\pi$ is a subquotient of the normalized induction of a supercuspidal representation $\pi_1\otimes \cdots \otimes \pi_t$, then take the sum of $\pi_i(p^r)$ over all $\pi_i$ which are unramified characters. This is the definition that we are going to use.
\end{rem}

\begin{thm}\label{MainTheoremGLn} Let $h\in C_c^{\infty}(\GL_n(\mathbb{Z}_p))$. Then $f_{n,p}\ast h$ and $\phi_h$ have matching (twisted) orbital integrals.
\end{thm}

\begin{rem} All Haar measures are normalized by giving a hyperspecial maximal compact subgroup measure $1$. For the case $h=e_{\GL_n(\mathbb{Z}_p)}$, this is exactly the usual base change identity, used e.g. in \cite{KottwitzLambdaAdic}.
\end{rem}

\begin{proof} Arguing as in the proof of Theorem \ref{BaseChangeIdentity}, it is enough to check that for any tempered irreducible representation $\pi$ of $\GL_n(\mathbb{Q}_p)$ with base-change lift $\Pi$, the equality
\[
\tr(f_{n,p}\ast h|\pi) = \tr((\phi_h,\sigma)|\Pi)
\]
holds. We begin with a computation of the right hand side.

Let $P_k$ be the standard parabolic with Levi $\GL_k\times \GL_{n-k}$ and let $N_k$ be its unipotent radical. For any admissible representation $\pi$ of $\GL_n(\mathbb{Q}_p)$ of finite length, let $\pi_{N_k}$ be its (unnormalized) Jacquet module with respect to $N_k$. Assume that
\[
\pi_{N_k} = \sum_{i=1}^{t_{\pi,k}} \pi_{N_k,i}^1\otimes \pi_{N_k,i}^2
\]
as elements of the Grothendieck group of representations of $\GL_k(\mathbb{Q}_p)\times \GL_{n-k}(\mathbb{Q}_p)$. Let $\Theta_{\Pi}$ be the distribution on $\GL_n(\mathbb{Q}_{p^r})$ given by $\Theta_{\Pi}(\delta)=\Theta_{\pi}(N\delta)$, where $\Theta_{\pi}$ is the character of $\pi$. Define $\Theta_{\Pi_{N_k,i}^1}$ in the same way. For any $f\in C_c^{\infty}(\GL_n(\mathbb{Q}_{p^r}))$, we will write $\Theta_{\Pi}(f)$ as $\tr( (f,\sigma) |\Pi )$, thinking of $\Pi$ as the base-change lift of $\pi$.

\begin{lem}\label{WeylIntegration} In this situation,
\[
\tr ( (\phi_h,\sigma)|\Pi ) = \sum_{k=1}^n\sum_{i=1}^{t_{\pi,k}} p^{(n-k) r} \tr( (\chi_{B_k},\sigma) | \Pi^1_{N_k,i} ) \tr( h | \mathrm{Ind}_{P_k(\mathbb{Z}_p)}^{\GL_n(\mathbb{Z}_p)} I_k^0\otimes \pi^2_{N_k,i} )\ ,
\]
where $\chi_{B_k}$ is the characteristic function of the set $B_k$ of all elements $\delta$ of $\GL_k(\mathbb{Q}_{p^r})\cap M_k(\mathbb{Z}_{p^r})$ with $v_p(\delta)=1$ such that $N\delta$ is elliptic.
\end{lem}

\begin{proof} Let $V_n^k$ be the set of $\delta\in \GL_n(\mathbb{Q}_{p^r})\cap M_n(\mathbb{Z}_{p^r})$ which are of height $k$. It is clearly enough to prove that
\[
\tr ( (\phi_h\chi_{V_n^k},\sigma)|\Pi ) = p^{(n-k) r} \sum_{i=1}^{t_{\pi,k}} \tr( (\chi_{B_k},\sigma) | \Pi^1_{N_k,i} ) \tr( h | \mathrm{Ind}_{P_k(\mathbb{Z}_p)}^{\GL_n(\mathbb{Z}_p)} I_k^0\otimes \pi^2_{N_k,i} )\ .
\]
But by the twisted Weyl integration formula, cf. \cite{ArthurClozel}, p. 36,
\[\begin{aligned}
&\tr ( (\phi_h \chi_{V_n^k},\sigma)|\Pi ) = \sum_{\substack{T_k\subset \GL_k, T_{n-k}\subset \GL_{n-k}\\ T_k\ {\rm anisotropic}}} |W(T_k\times T_{n-k},\GL_k\times \GL_{n-k})|^{-1}\\
&\times \int_{T_k(\mathbb{Q}_{p^r})^{1-\sigma}\times T_{n-k}(\mathbb{Q}_{p^r})^{1-\sigma}\backslash T_k(\mathbb{Q}_{p^r})_1\times T_{n-k}(\mathbb{Z}_{p^r})} \Delta_{\GL_n(\mathbb{Q}_p)}^2 (Nt) TO_{t\sigma}(\phi_h) \Theta_{\pi}(Nt)\ .
\end{aligned}\]
Here $T_k(\mathbb{Q}_{p^r})_1$ denotes the set of elements $t\in T_k(\mathbb{Q}_{p^r})$ with $v_p(\det t)=1$. Also recall that
\[
\Delta_{\GL_n(\mathbb{Q}_p)}^2 (t) = |\det(1-\mathrm{Ad}\ t|\mathrm{Lie}\ \mathfrak{gl}_n)|_p\ .
\]

We can make several simplifications. First of all, writing $t=(t_1,t_2)\in T_k(\mathbb{Q}_{p^r})\times T_{n-k}(\mathbb{Q}_{p^r})$,
\[
\Theta_{\pi}(Nt) = \Theta_{\pi_{N_k}}(Nt) = \sum_{i=1}^{t_{\pi,k}} \Theta_{\pi^1_{N_k,i}}(Nt_1) \Theta_{\pi^2_{N_k,i}}(Nt_2)
\]
by a result of Casselman, \cite{CasselmanJacquet}. We may use this result because of the conditions $t_1\in T_k(\mathbb{Q}_{p^r})_1$ with $T_k$ anisotropic and $t_2\in T_{n-k}(\mathbb{Z}_{p^r})$.

Second,
\[
\Delta_{\GL_n(\mathbb{Q}_p)}^2(Nt) = p^{(n-k)r} \Delta_{\GL_k(\mathbb{Q}_p)}^2(Nt_1) \Delta_{\GL_{n-k}(\mathbb{Q}_p)}^2(Nt_2)\ ,
\]
because the action of $\mathrm{Ad}\ Nt$ on the Lie algebra of $N_k$ is by multiplication through $p$, and on the opposite nilpotent Lie algebra by multiplication through $p^{-1}$.

Third, we have the following lemma.

\begin{lem} Let $\phi\in C_c^{\infty}(\GL_n(\mathbb{Q}_{p^r}))$ have support in the elements which are of height $k$ and assume that $\phi$ is invariant under $\GL_n(\mathbb{Z}_{p^r})$-$\sigma$-conjugation. Then for $t=(t_1,t_2)\in \GL_k(\mathbb{Q}_{p^r})\times \GL_{n-k}(\mathbb{Z}_{p^r})$ with $Nt_1$ elliptic, $v_p(\det t_1)=1$, we have
\[
TO_{t\sigma}(\phi) = TO_{t\sigma}^{\GL_k\times \GL_{n-k}}(\phi)\ ,
\]
where the right-hand side is the twisted orbital integral on the Levi subgroup $\GL_k\times \GL_{n-k}$.
\end{lem}

\begin{proof} We apply Proposition 3.12 of \cite{ArthurClozel} with $P$ being the opposite parabolic of $P_k$. We have to check that the constant term $\phi^P$ of $\phi$ along $P$ equals $\phi$. Let $N$ be the nilpotent radical of $P$. It is enough to see that for any $\GL_k(\mathbb{Q}_{p^r})\times \GL_{n-k}(\mathbb{Q}_{p^r})$-$\sigma$-conjugate $t^{\prime}$ of $t$ and $n\in N(\mathbb{Q}_{p^r})$,
\[
\phi(t^{\prime}n)=\left\{\begin{array}{ll} \phi(t^{\prime}) & n\in N(\mathbb{Z}_{p^r}) \\ 0 & \mathrm{else}\ . \end{array}\right .
\]
This holds true, because both sides vanish unless $t^{\prime}\in (\GL_k(\mathbb{Q}_{p^r})\cap M_k(\mathbb{Z}_{p^r}))\times \GL_{n-k}(\mathbb{Z}_{p^r})$ and $n\in N(\mathbb{Z}_{p^r})$, in which case, $t^{\prime}$ and $t^{\prime}n$ are easily seen to be $\GL_n(\mathbb{Z}_{p^r})$-$\sigma$-conjugate.
\end{proof}

If one defines $h_{n-k}\in C_c^{\infty}(\GL_{n-k}(\mathbb{Z}_{p^r}))$ by
\[
h_{n-k}(\delta_2) = \tr(h\times N\delta_2 | I_k)\ ,
\]
then inserting $\phi=\phi_h \chi_{V_n^k}$ in this lemma implies that for $t=(t_1,t_2)\in T_k(\mathbb{Q}_{p^r})_1\times T_{n-k}(\mathbb{Z}_{p^r})$ as above,
\[
TO_{t\sigma}(\phi_h)=TO_{t_1\sigma}(\chi_{B_k})TO_{t_2\sigma}(h_{n-k})\ .
\]

Now, it is clear that the Weyl integration formula factors as a sum over $i=1,\ldots,t_{\pi,k}$ of products of the Weyl integration formulas for $\GL_k(\mathbb{Q}_{p^r})$ and $\GL_{n-k}(\mathbb{Q}_{p^r})$. To finish the proof, we only need to show that
\[
\tr( h_{n-k} | \Pi_{n-k} ) = \tr( h | \mathrm{Ind}_{P_k(\mathbb{Z}_p)}^{\GL_n(\mathbb{Z}_p)} I_k^0\otimes \pi_{n-k} )
\]
for any irreducible smooth representation $\pi_{n-k}$ of $\GL_{n-k}(\mathbb{Q}_p)$.  If $h_{n-k}^{\prime}\in C_c^{\infty}(\GL_{n-k}(\mathbb{Z}_p))$ is defined by
\[
h_{n-k}^{\prime}(\gamma_2) = \tr(h\times \gamma_2 | I_k)\ ,
\]
then $\tr( h_{n-k} | \Pi_{n-k} ) = \tr( h_{n-k}^{\prime} | \pi_{n-k} )$ by Corollary \ref{BCUnit} and Proposition \ref{SigmaConjClasses}, taking $M=M_n(\mathbb{Z}_p)$ and $I$ small. We need to see that
\[
\tr( h_{n-k}^{\prime} | \pi_{n-k} ) = \tr( h | \mathrm{Ind}_{P_k(\mathbb{Z}_p)}^{\GL_n(\mathbb{Z}_p)} I_k^0\otimes \pi_{n-k} )\ .
\]
But for this we can replace $\mathbb{Z}_p$ by $\mathbb{Z}/p^m\mathbb{Z}$ for $m$ large and $\pi_{n-k}$ by an irreducible representation of $\GL_{n-k}(\mathbb{Z}/p^m\mathbb{Z})$. Now, the computation is easy, using the orthogonality relations:
\[\begin{aligned}
&\ \# \GL_{n-k}(\mathbb{Z}/p^m\mathbb{Z}) \tr( h_{n-k}^{\prime} | \pi_{n-k} ) \\
=&\ \sum_{\gamma_2} \tr(h\times \gamma_2 | \mathrm{Ind}_{P_k(\mathbb{Z}/p^m\mathbb{Z})}^{\GL_n(\mathbb{Z}/p^m\mathbb{Z})} I_k^{0,\Gamma(p^m)}\otimes \mathbb{C}[\GL_{n-k}(\mathbb{Z}/p^m\mathbb{Z})]) \tr(\gamma_2|\pi_{n-k}) \\
=&\ \sum_{\gamma_2} \tr(h\times \gamma_2 | \mathrm{Ind}_{P_k(\mathbb{Z}/p^m\mathbb{Z})}^{\GL_n(\mathbb{Z}/p^m\mathbb{Z})} I_k^{0,\Gamma(p^m)}\otimes (\bigoplus_{\pi_{n-k}^{\prime}} \pi_{n-k}^{\prime}\otimes \pi_{n-k}^{\prime\vee})) \tr(\gamma_2|\pi_{n-k}) \\
=&\ \# \GL_{n-k}(\mathbb{Z}/p^m\mathbb{Z}) \tr(h | \mathrm{Ind}_{P_k(\mathbb{Z}/p^m\mathbb{Z})}^{\GL_n(\mathbb{Z}/p^m\mathbb{Z})} I_k^{0,\Gamma(p^m)}\otimes \pi_{n-k})\ ,
\end{aligned}\]
where $\pi_{n-k}^{\prime}$ runs through irreducible representations of $\GL_{n-k}(\mathbb{Z}/p^m\mathbb{Z})$.
\end{proof}

\begin{lem}\label{CompEllTrace} Let $\pi$ be an irreducible smooth representation of $\GL_n(\mathbb{Q}_p)$ which is a subquotient of a parabolic induction from a supercuspidal representation $\pi_1\otimes \cdots \otimes \pi_t$ with no $\pi_i$ being an unramified character. Then $\tr( (\chi_{B_n},\sigma)|\Pi )=0$.

If $\pi$ is the trivial representation of $\GL_n(\mathbb{Q}_p)$, then $\tr( (\chi_{B_n},\sigma)|\Pi )=1$.
\end{lem}

\begin{proof} Consider the first assertion. Put $h=e_{\GL_n(\mathbb{Z}_p)}$ in Lemma \ref{WeylIntegration}. The left hand side vanishes, because $\phi_h$ is bi-$\GL_n(\mathbb{Z}_{p^r})$-invariant and $\pi$, hence $\Pi$, is not unramified, e.g. by Theorem \ref{BaseChangeIdentity} for $\mathcal{G}=\GL_n$. Further, for all $k<n$, we have by induction $\tr( (\chi_{B_k},\sigma)|\Pi_{N_k,i}^1 )=0$. Hence also the term for $k=n$ vanishes, which gives exactly the desired identity.

The same inductive computation also works if $\pi$ is the trivial representation, using that the left hand side is now the volume of
\[
\GL_n(\mathbb{Z}_{p^r})\mathrm{diag}(p,1,\ldots,1)\GL_n(\mathbb{Z}_{p^r})\ ,
\]
which can be determined to be $1+p^r+\ldots+p^{(n-1)r}$, e.g. by interpreting it as the number of neighbors of the vertex corresponding to $\GL_n(\mathbb{Z}_{p^r})$ in the building of $\mathrm{PGL}_n$.
\end{proof}

It is enough to show that for all tempered irreducible representations $\pi$ of $\GL_n(\mathbb{Q}_p)$
\begin{equation}\label{EqRepr}
p^{\frac{n-1}2 r}\mathrm{tr}^{\mathrm{ss}}(\Phi_p^r | \sigma_{\pi}) \pi = \sum_{k=1}^n \sum_{i=1}^{t_{\pi,k}} p^{(n-k)r} \tr( (\chi_{B_k},\sigma) | \Pi^1_{N_k,i} ) \mathrm{Ind}_{P_k(\mathbb{Z}_p)}^{\GL_n(\mathbb{Z}_p)} I_k^0\otimes \pi^2_{N_k,i}
\end{equation}
as virtual representations of $\GL_n(\mathbb{Z}_p)$. We prove the theorem by induction on $n$, knowing that the equation \eqref{EqRepr} holds true for all irreducible smooth representations of $\GL_{n^{\prime}}(\mathbb{Q}_p)$ with $n^{\prime}<n$.

\begin{lem}\label{InducedOK} Assume that $\pi$ is not necessarily irreducible, but is the parabolic induction of an irreducible smooth representation. Then equation \eqref{EqRepr} holds true.
\end{lem}

\begin{proof} We first remark that the equation \eqref{EqRepr} makes sense for $\pi$ of this form, because by Remark \ref{DefSemisimpleTrace},
\[
\mathrm{tr}^{\mathrm{ss}}(\Phi_p^r | \sigma_{\pi})
\]
has a definitive meaning. Assume that $\pi=\mathrm{Ind}_{P_m(\mathbb{Q}_p)}^{\GL_n(\mathbb{Q}_p)} \pi^{\prime}\otimes \pi^{\prime\prime}$, where $\mathrm{Ind}$ denotes unnormalized induction. We remark that in this case, we have equalities in the Grothendieck group of admissible representations:
\[\begin{aligned}
\pi=\mathrm{Ind}_{P_m(\mathbb{Q}_p)}^{\GL_n(\mathbb{Q}_p)} \pi^{\prime}\otimes \pi^{\prime\prime}&=\text{n-Ind}_{P_m(\mathbb{Q}_p)}^{\GL_n(\mathbb{Q}_p)} \pi^{\prime}[\tfrac{m-n}2]\otimes \pi^{\prime\prime}[\tfrac m2] \\
&=\mathrm{Ind}_{P_m(\mathbb{Q}_p)}^{\GL_n(\mathbb{Q}_p)} \pi^{\prime\prime}[m]\otimes \pi^{\prime}[m-n]\ ,
\end{aligned}\]
where $[x]$ means twisting by $|\det|_p^x$ and $\text{n-Ind}$ denotes normalized induction. In particular,
\[
\mathrm{tr}^{\mathrm{ss}}(\Phi_p^r | \sigma_{\pi}) = p^{\frac{m-n}2} \mathrm{tr}^{\mathrm{ss}}(\Phi_p^r | \sigma_{\pi^{\prime}}) + p^{\frac m2} \mathrm{tr}^{\mathrm{ss}}(\Phi_p^r | \sigma_{\pi^{\prime\prime}})\ .
\]
Further, assume by induction that the equation \eqref{EqRepr} holds true for $\pi^{\prime}$ and $\pi^{\prime\prime}$. Note that by the restriction induction formula of Bernstein-Zelevinsky, \cite{BernsteinZelevinsky}, Lemma 2.12, we have
\[
\pi_{N_k}=\mathrm{Ind}_{(P_m\cap \GL_k\times \GL_{n-k})(\mathbb{Q}_p)}^{\GL_k\times \GL_{n-k}(\mathbb{Q}_p)} \pi^{\prime}_{N_k}\otimes \pi^{\prime\prime} +  \mathrm{Ind}_{(P_{n-m}\cap \GL_k\times \GL_{n-k})(\mathbb{Q}_p)}^{\GL_k\times \GL_{n-k}(\mathbb{Q}_p)} \pi^{\prime\prime}_{N_k}[m]\otimes \pi^{\prime}[m-n] + R\ ,
\]
where the rest $R$ is a sum of representations of the form $\pi_k\otimes \pi_{n-k}$ with $\pi_k$ properly induced. Let
\[
\pi^{\prime}_{N_k} = \sum_{i=1}^{t_{\pi^{\prime},k}} \pi^{\prime 1}_{N_k,i}\otimes \pi^{\prime 2}_{N_k,i}\ .
\]
By induction, we have an equality of virtual representations of $\GL_m(\mathbb{Z}_p)$:
\[\begin{aligned}
&\ p^{\frac{m-1}2 r}\mathrm{tr}^{\mathrm{ss}}(\Phi_p^r | \sigma_{\pi^{\prime}}) \pi^{\prime}\\
=&\ \sum_{k=1}^m \sum_{i=1}^{t_{\pi^{\prime},k}} p^{(m-k)r} \tr( (\chi_{B_k},\sigma) | \Pi^{\prime 1}_{N_k,i} ) \mathrm{Ind}_{(P_k\cap \GL_m)(\mathbb{Z}_p)}^{\GL_m(\mathbb{Z}_p)} I_k^0\otimes \pi^{\prime 2}_{N_k,i}\ .
\end{aligned}\]
Hence, defining $P_{k,m}$ to be the parabolic with breakpoints $k\leq m$,
\[\begin{aligned}
&\ p^{\frac{2n-m-1}2 r}\mathrm{tr}^{\mathrm{ss}}(\Phi_p^r | \sigma_{\pi^{\prime}}) \pi \\
=&\ \sum_{k=1}^m \sum_{i=1}^{t_{\pi^{\prime},k}} p^{(n-k)r} \tr( (\chi_{B_k},\sigma) | \Pi^{\prime 1}_{N_k,i} ) \mathrm{Ind}_{P_{k,m}(\mathbb{Z}_p)}^{\GL_n(\mathbb{Z}_p)} I_k^0\otimes \pi^{\prime 2}_{N_k,i}\otimes \pi^{\prime\prime} \\
=&\ \sum_{k=1}^m \sum_{i=1}^{t_{\pi^{\prime},k}} p^{(n-k)r} \tr( (\chi_{B_k},\sigma) | \Pi^{1}_{N_k,i} ) \mathrm{Ind}_{P_k(\mathbb{Z}_p)}^{\GL_n(\mathbb{Z}_p)} I_k^0\otimes \pi^{2}_{N_k,i}\ ,
\end{aligned}\]
as virtual representations of $\GL_n(\mathbb{Z}_p)$, where $\pi^1_{N_k,i}=\pi^{\prime 1}_{N_k,i}$ and $\pi^2_{N_k,i}=\mathrm{Ind} \pi^{\prime 2}_{N_k,i}\otimes \pi^{\prime\prime}$. This definition is useful, since
\[
\mathrm{Ind}_{(P_m\cap \GL_k\times \GL_{n-k})(\mathbb{Q}_p)}^{\GL_k\times \GL_{n-k}(\mathbb{Q}_p)} \pi^{\prime}_{N_k}\otimes \pi^{\prime\prime} = \sum_{i=1}^{t_{\pi^{\prime},k}} \pi^1_{N_k,i}\otimes \pi^2_{N_k,i}
\]
expresses the first part of $\pi_{N_k}$ for $k\leq m$. Note that for $k>m$,
\[
\mathrm{Ind}_{(P_m\cap \GL_k\times \GL_{n-k})(\mathbb{Q}_p)}^{\GL_k\times \GL_{n-k}(\mathbb{Q}_p)} \pi^{\prime}_{N_k}\otimes \pi^{\prime\prime}
\]
is also of the form $\pi_k\otimes \pi_{n-k}$ with $\pi_k$ properly induced. Repeating this with $\pi^{\prime}$ and $\pi^{\prime\prime}$ exchanged, we are left to show that any representation of the form $\pi_k\otimes \pi_{n-k}$ in $\pi_{N_k}$ with $\pi_k$ properly induced contributes trivially to \eqref{EqRepr}. This follows from the fact that if $\pi_k$ is properly induced, then $\tr((\chi_{B_k},\sigma)|\Pi_k)=0$, as $\chi_{B_k}$ is supported in elements whose norm is elliptic, and the character of $\pi_k$ vanishes on elliptic elements.
\end{proof}

Further, the equation \eqref{EqRepr} is trivial for representations which are subquotients of the parabolic induction of a supercuspidal representation $\pi_1\otimes \cdots \otimes \pi_t$ with no $\pi_i$ being an unramified character, as both sides vanish. Hence we are left to check it for an unramified twist of the Steinberg representation, or equivalently by Lemma \ref{InducedOK} for an unramified character $\chi\circ \det$. The equation \eqref{EqRepr} reduces to
\[
(1+p^r+\ldots+p^{(n-1)r})\chi(p^r) (\chi\circ \det) = \sum_{k=1}^n p^{(n-k)r} \chi(p^r) \mathrm{Ind}_{P_k(\mathbb{Z}_p)}^{\GL_n(\mathbb{Z}_p)} I_k^0\otimes (\chi\circ \det)\ .
\]
We see that we may assume $\chi=1$. Then this is a trivial consequence of the definition of $I_k^0$ and the long exact sequence in Lemma \ref{CombSteinberg}.
\end{proof}

\section{Comparison with the Arthur-Selberg trace formula}

We now conclude the proof of Theorem D, i.e. the determination of the semisimple local factor of the Shimura varieties $\mathrm{Sh}_K$ at the place $\mathfrak{p}$ in terms of $L$-functions of automorphic forms, cf. the introduction for the precise statement.

The main ingredients are the work of Kottwitz on the number of points of the Shimura variety over finite fields, cf. \cite{KottwitzPoints}, with the refinements in the case at hand in \cite{KottwitzLambdaAdic}, and the Arthur trace formula, which is very simple in our case, because our Shimura varieties are proper. Let
\[
H^{\ast}=\sum (-1)^i H^i_{\mathrm{et}}(\mathrm{Sh}_K\otimes_{k} \overline{\mathbb{Q}}_p,\mathbb{Q}_{\ell})
\]
in the Grothendieck group of representations of $\mathrm{Gal}(\overline{\mathbb{Q}}_p/\mathbb{Q}_p)$. Write $f^p = e_{K^p}$ and $\tilde{h}_p = e_{K_p}$. We know that $K_p$ has the form
\[
K_p=K_p^0\times \mathbb{Z}_p^{\times}\subset \GL_n(\mathbb{Q}_p)\times \mathbb{Q}_p^{\times}\cong \mathbf{G}(\mathbb{Q}_p)\ .
\]
We may assume that $K_p^0\subset \GL_n(\mathbb{Z}_p)$. Let $h_p=e_{K_p^0}$, so that $h_p\in C_c^{\infty}(\GL_n(\mathbb{Z}_p))$. Note that $h_p=h_p^{\vee}$, where we recall that by definition $h_p^{\vee}(g)=h_p((g^{-1})^t)$.

\begin{lem}\label{ReduceToGLn} Let $f_{p,h_p}$ be the function defined by
\[\begin{aligned}
f_{p,h_p}: \mathbf{G}(\mathbb{Q}_p)\cong \GL_n(\mathbb{Q}_p)\times \mathbb{Q}_p^{\times}&\longrightarrow \mathbb{C} \\
(g,x)&\longmapsto \left\{\begin{array}{ll} (f_{n,p}\ast h_p)((g^{-1})^t) & v_p(x)=-r \\ 0 & \mathrm{else}\ .\end{array}\right .
\end{aligned}\]
Then for any irreducible smooth representation $\pi_p$ of $\mathbf{G}(\mathbb{Q}_p)$,
\[
p^{\frac{n-1}2 r}\tr(\tilde{h}_p|\pi_p)\mathrm{tr}^{\mathrm{ss}}(\Phi_p^r | r\circ \sigma_{\pi_p}) = \tr(f_{p,h_p}|\pi_p)\ .
\]
\end{lem}

\begin{rem} We apologize for the two different uses of $r$ in this paper.
\end{rem}

\begin{proof} Let $e_{p^{-r}\mathbb{Z}_p^{\times}}$ be the characteristic function of $p^{-r}\mathbb{Z}_p^{\times}$ divided by its volume.  Since $\pi_p$ is an irreducible smooth representation of $\GL_n(\mathbb{Q}_p)\times \mathbb{Q}_p^{\times}$, it can be written as $\pi_p=\pi_p^0\otimes \chi_{\pi_p}$ for some irreducible smooth representation $\pi_p^0$ of $\GL_n(\mathbb{Q}_p)$ and a character $\chi_{\pi_p}$ of $\mathbb{Q}_p^{\times}$. We compute
\[\begin{aligned}
\tr(f_{p,h_p}|\pi_p)&= \tr(f_{n,p}\ast h_p|\pi_p^{0\vee}) \tr(e_{p^{-r}\mathbb{Z}_p^{\times}}|\chi_{\pi_p}) \\
&=p^{\frac{n-1}2 r}\mathrm{tr}^{\mathrm{ss}}(\Phi_p^r | \sigma_{\pi_p^0}^{\vee}) \tr(h_p|\pi_p^0) \tr(e_{p^{-r}\mathbb{Z}_p^{\times}}|\chi_{\pi_p}) \\
&= p^{\frac{n-1}2 r}\tr(\tilde{h}_p|\pi_p)\mathrm{tr}^{\mathrm{ss}}(\Phi_p^r | r\circ \sigma_{\pi_p})\ .
\end{aligned}\]
\end{proof}

The equation
\[
\zeta_{\mathfrak{p}}^{\mathrm{ss}}(\mathrm{Sh}_K,s) = \prod_{\pi_f} L^{\mathrm{ss}}(s-\tfrac{n-1}{2},\pi_p,r)^{a(\pi_f)\mathrm{dim} \pi_f^K}
\]
we want to prove reduces by standard methods to showing that for all $r\geq 1$ the equation
\[
\mathrm{tr}^{\mathrm{ss}}(\Phi_{p^r}|H^{\ast}) = \sum_{\pi_f=\pi_f^p\otimes \pi_p} p^{\frac{n-1}2 r}a(\pi_f)\tr(f^p|\pi_f^p)\tr(\tilde{h}_p|\pi_p) \mathrm{tr}^{\mathrm{ss}}(\Phi_p^r | r\circ \sigma_{\pi_p})
\]
holds true: Indeed, take the logarithms of both sides and use the Lefschetz trace formula for the left-hand side. Using Lemma \ref{ReduceToGLn}, this is equivalent to
\[
\mathrm{tr}^{\mathrm{ss}}(\Phi_{p^r}|H^{\ast}) = \sum_{\pi_f=\pi_f^p\otimes \pi_p} a(\pi_f)\tr(f^p|\pi_f^p) \tr(f_{p,h_p}|\pi_p)\ .
\]

This equation is proved in exactly the same way as the expression \cite[(5.4)]{KottwitzLambdaAdic} for the corresponding trace on a Galois representation. We just note the necessary changes in the argument.

First, a modification has to be done in Section 16 of \cite{KottwitzPoints}, where the number of points within one $\mathbb{F}_{p^r}$-isogeny class is computed. The function $\phi_r$ occuring there counts the number of lattices $\Lambda$ that give rise to $\mathbb{F}_{p^r}$-points of the moduli problem. More concretely, the set of such lattices is in bijection to
\[
Y_p=\{x\in \mathbf{G}(\mathbb{Q}_{p^r})/\mathbf{G}(\mathbb{Z}_{p^r})\mid x^{-1}\delta x^{\sigma}\in \mathbf{G}(\mathbb{Z}_{p^r})\mu(p^{-1})\mathbf{G}(\mathbb{Z}_{p^r})\}\ ,
\]
cf. \cite{KottwitzPoints}, p.432, noting that in our case $\mu(p^{-1})$ is defined over $\mathbb{Q}_p$, hence $\sigma$ acts trivially. But note that then $x^{-1}\delta x^{\sigma}$ up to $\mathbf{G}(\mathbb{Z}_{p^r})$-$\sigma$-conjugation is exactly the refined $\delta$ of the corresponding point of $\mathcal{M}(\mathbb{F}_{p^r})$ as established in Lemma \ref{CompDelta}. Hence, when computing, using Theorem \ref{CalcSemisimpleTrace},
\[\begin{aligned}
\mathrm{tr}^{\mathrm{ss}}(\Phi_{p^r}|H^{\ast}) &= \sum_{x\in \mathcal{M}(\mathbb{F}_{p^r})} \mathrm{tr}^{\mathrm{ss}}(\Phi_{p^r}\times h_p|(R\psi\mathcal{F}_{\infty})_{\overline{x}})\\
&= \sum_{x\in \mathcal{M}(\mathbb{F}_{p^r})} \phi_{h_p}(\delta_{x,0})\ ,
\end{aligned}\]
we have to weight the corresponding point with the factor $\phi_{h_p}(\delta_{x,0})$. As the second component of $\delta_x$ always has valuation $-1$, this shows that we get the correct result if we replace $\phi_r$ with $\phi_{h_p}^{\vee}\times e_{p^{-1}\mathbb{Z}_{p^r}^{\times}}$ in \cite{KottwitzPoints}, where $\phi_{h_p}^{\vee}(g) = \phi_{h_p}((g^{-1})^t)$.

We simply record the intermediate result.

\begin{prop} The semisimple Lefschetz number $\mathrm{tr}^{\mathrm{ss}}(\Phi_p^r | H^{\ast})$ equals
\[
\sum_{\gamma_0} \sum_{\gamma,\delta} c(\gamma_0;\gamma,\delta) O_{\gamma}(f^p)TO_{\delta\sigma}(\phi_{h_p}^{\vee}\times e_{p^{-1}\mathbb{Z}_{p^r}^{\times}})\ ,
\]
in the notation of \cite{KottwitzPoints}, p.442. In particular, $c(\gamma_0;\gamma,\delta)$ is a certain volume factor.
\end{prop}

With this modification, the argument of \cite{KottwitzLambdaAdic} works, if instead of the base-change fundamental lemma at the end of p.662, one uses Theorem \ref{MainTheoremGLn}. Note that the extra $\mathbb{G}_m$-component of $\delta$ causes no trouble. This finishes the proof of Theorem D.

As a last point, we give a reformulation of the last proposition. Let $\phi_{n,p}$ be the function of the Bernstein center of $\GL_n(\mathbb{Q}_{p^r})$ that acts on any irreducible smooth representation $\Pi$ of $\GL_n(\mathbb{Q}_{p^r})$ through the scalar
\[
p^{\frac{n-1}2 r} \mathrm{tr}^{\mathrm{ss}}(\Phi_{p^r} | \sigma_{\Pi})\ .
\]
Its existence is proved as usual and it is readily checked that for any tempered irreducible representation $\pi$ of $\GL_n(\mathbb{Q}_p)$ with base-change lift $\Pi$, the scalar through which $f_{n,p}$ acts on $\pi$ is the same as the scalar through which $\phi_{n,p}$ acts on $\Pi$. Indeed, it is enough to do this for representations unitarily induced from supercuspidal (which are preserved under base change), and use that representations cannot become unramified after unramified base change if they were not from the start, cf. proof of Lemma 10.2 in \cite{Scholze}.

Fix a group scheme $\mathcal{G}=\mathcal{G}_{M,I}$ over $\mathbb{Z}_p$ with generic fibre $\GL_n$ as in Section 2 and assume $K_p^0=\mathcal{G}(\mathbb{Z}_p)$. Noting that the function
\[
\phi_{\mathcal{G},r} = (\phi_{n,p}\ast e_{\mathcal{G}(\mathbb{Z}_{p^r})})^{\vee}\times e_{p^{-1}\mathbb{Z}_{p^r}^{\times}}
\]
lies in the center of the Hecke algebra $\mathcal{H}(\mathbf{G}(\mathbb{Q}_{p^r}),\mathcal{G}(\mathbb{Z}_{p^r})\times \mathbb{Z}_{p^r}^{\times})$, the following theorem proves a conjecture of Haines and Kottwitz in the case at hand.

\begin{thm}\label{HainesKottwitz} The semisimple Lefschetz number $\mathrm{tr}^{\mathrm{ss}}(\Phi_p^r | H^{\ast})$ equals
\[
\sum_{\gamma_0} \sum_{\gamma,\delta} c(\gamma_0;\gamma,\delta) O_{\gamma}(f^p)TO_{\delta\sigma}(\phi_{\mathcal{G},r})\ .
\]
\end{thm}

\begin{proof} We have seen that this is correct if we replace $TO_{\delta\sigma}(\phi_{\mathcal{G},r})$ by $TO_{\delta\sigma}(\phi_{h_p}^{\vee}\times e_{p^{-1}\mathbb{Z}_{p^r}^{\times}})$. However, both functions have matching (twisted) orbital integrals with $(f_{n,p}\ast e_{\mathcal{G}(\mathbb{Z}_p)})^{\vee}\times e_{p^{-r}\mathbb{Z}_p^{\times}}$: For the first function, this is a consequence of Theorem \ref{BaseChangeIdentity}, and for the second function, it follows from Theorem \ref{MainTheoremGLn}. Therefore the twisted orbital integrals agree, as desired.
\end{proof}

\bibliographystyle{abbrv}
\bibliography{GLn}

\end{document}